\documentclass[10pt]{amsart}
\usepackage{latexsym,amssymb,graphics,epsfig,xcolor}

\theoremstyle{plain}
  \newtheorem{theorem}{Theorem}[section]
  \newtheorem{proposition}[theorem]{Proposition}
  
  \newtheorem{corollary}[theorem]{Corollary}
  
\theoremstyle{definition}
  \newtheorem{definition}[theorem]{Definition}
  \newtheorem{example}[theorem]{Example}

 \theoremstyle{remark}
  \newtheorem{remark}[theorem]{Remark}

\numberwithin{equation}{section}

\def\NN{{\mathbb N}}
\def\ZZ{{\mathbb Z}}
\def\QQ{{\mathbb Q}}
\def\RR{{\mathbb R}}

\def\LLL{{\mathcal{L}}}
\def\OOO{{\mathcal{O}}}
\def\CCC{{\mathcal{C}}}
\def\AAA{{\mathcal{A}}}

\def\symm{\mathfrak{S}}

\def\Hilb{\mathrm{Hilb}}

\def\Sym{\mathrm{Sym}}

\def\Kroot{K^\mathrm{root}}
\def\Kweight{K^\mathrm{wt}}
\def\Lroot{L^\mathrm{root}}
\def\Vroot{V^\mathrm{root}}
\def\Lweight{L^\mathrm{wt}}

\def\sss{\mathrm{s}}
\def\HHH{\mathrm{H}}
\def\Des{\mathrm{Des}}
\def\maj{\mathrm{maj}}

\def\Irr{\mathrm{Irr}}

\def\xx{\mathbf{x}}
\def\yy{\mathbf{y}}
\def\aa{\mathbf{a}}
\def\bb{\mathbf{b}}
\def\XX{\mathbf{X}}

\def\uu{\mathbf{u}}

\begin{document}

\title[Linear extension sums as valuations on cones]
{Linear extension sums as valuations on cones}

\author{Adrien Boussicault}
\email{adrien.boussicault@univ-mlv.fr}
\address{LaBRI\\
Universit\'e Bordeaux 1\\
351 Cours de la Lib\'eration\\
33400 Talence\\
France}

\author{Valentin F\'eray}
\email{feray@labri.fr}
\address{LaBRI\\
Universit\'e Bordeaux 1\\
351 Cours de la Lib\'eration\\
33400 Talence\\
France}

\author{Alain Lascoux}
\email{alain.lascoux@univ-mlv.fr}
\address{Universit\'e Paris-Est\\
Institut Gaspard Monge\\
77454 Marne-la-Vall\'ee\\
France}

\author{Victor Reiner}
\email{reiner@math.umn.edu}
\address{School of Mathematics\\
University of  Minnesota\\
Minneapolis, MN 55455\\
USA}

\thanks{Fourth author supported by NSF grant DMS-0601010.}

\keywords{poset, rational function identities, 
valuation of cones, lattice points, 
affine semigroup ring, Hilbert series, total residue, 
root system, weight lattice}

\subjclass{
06A11
52B45, 
52B20,
}

\begin{abstract}
The geometric and algebraic theory of valuations on cones
is applied to understand identities involving summing
certain rational functions over the set of linear extensions of a poset.
\end{abstract}

\maketitle


\section{Introduction}
\label{Intro-section}

This paper presents a different viewpoint on the following two
classes of rational function summations, which are both summations
over the set $\LLL(P)$ of all linear extensions
of a partial order $P$ on the set 
$\{1,2,\ldots,n\}$:
$$
\begin{aligned}
\Psi_P(\xx)&:= \sum_{w \in \LLL(P)} 
  w\left(\frac{1}{(x_1-x_2)(x_2-x_3)\cdots(x_{n-1}-x_n)} \right); \\
\Phi_P(\xx)&:= \sum_{w \in \LLL(P)} 
  w\left(\frac{1}{x_1(x_1+x_2)(x_1+x_2+x_3)\cdots(x_1+\cdots+x_n)} \right).
\end{aligned}
$$
Recall that a {\it linear extension} is a permutation 
$w=(w(1),\ldots,w(n))$ in the symmetric group $\symm_n$  for which
the {\it linear order} $P_w$ defined by 
$
w(1) <_{P_w} \cdots <_{P_w} w(n)
$
satisfies $i <_{P_w} j$ whenever $i <_P j$.

Several known results express these sums explicitly for particular posets $P$ as 
rational functions in lowest terms.  In the past, these results have most often been
proven by induction, sometimes in combination with techniques 
such as {\it divided differences} and more general operators 
on multivariate polynomials.
We first explain three of these results that motivated us.

\subsection{Strongly planar posets}

The rational function $\Psi_P(\xx)$ was introduced by C. Greene \cite{Greene}
in his work on the Murnaghan-Nakayama formula.  There he evaluated
$\Psi_P(\xx)$ when $P$ is a {\it strongly planar} poset in the
sense that the poset $P\sqcup \{\hat{0},\hat{1}\}$ with an extra bottom
and top element has a planar embedding for its Hasse diagram, with all edges
directed upward in the plane.  
To state his evaluation, note that in this situation, the edges of the Hasse diagram
for $P$ dissect the plane into bounded regions $\rho$, and the set of vertices
lying on the boundary of $\rho$ will consist of two chains,
having a common minimum element $\min(\rho)$ and maximum $\max(\rho)$ element
in the partial order $P$.

\vskip .1in
\noindent
{\bf Theorem A.} (Greene \cite[Theorem 3.3]{Greene})
{\it 
For any strongly planar poset $P$,
$$
\Psi_P(\xx) = \frac{ \prod_{\rho} (x_{\min(\rho)} - x_{\max(\rho)}) }
                      {  \prod_{i \lessdot_P j} (x_i - x_j)  }
$$
where the product in the denominator runs over all covering relations
$i \lessdot_P j$, or over the edges of the Hasse diagram for $P$,
while the product in the numerator runs over all bounded regions $\rho$ 
for the Hasse diagram for $\rho$.
}

\subsection{Skew diagram posets}

Further work on $\Psi_P(\xx)$ appeared in 
\cite{Boussicault, TheseBoussicault, BoussicaultFeray, Ilyuta}.
For example, we will prove in Section~\ref{skew-diagram-section} the following generalization
of a result of the first author.
Consider a {\it skew (Ferrers) diagrams} $D=\lambda/\mu$,
in English notation as a collection of points $(i,j)$ in the plane,
where rows are numbered $1,2,\ldots,r$ from top to bottom
(the usual English convention), and the columns
numbered $1,2,\ldots,c$ from right to left (not the usual English convention).
Thus the northeasternmost and southwesternmost points of $D$ are labelled $(1,1)$ 
and $(r,c)$, respectively;  see Example~\ref{skew-example}.  
Define the {\it bipartite poset} $P_D$ on the set $\{x_1,\ldots,x_r,y_1,\ldots,y_c\}$
having an order relation $x_i <_{P_D} y_j$ whenever $(i,j)$ is a point of $D$.

\vskip .1in
\noindent
{\bf Theorem B.} 
{\it
For any skew diagram $D$,
$$
\Psi_{P_D}(\xx) = \frac{ \sum_{\pi} \prod_{(i,j) \in D \setminus \pi} (x_i - y_j) }
                      {  \prod_{(i,j) \in D} (x_i - y_j)  }.
$$
where the product in the numerator runs over all lattice paths $\pi$ from $(1,1)$ to $(r,c)$
inside $D$ that take steps either one unit south or west.

In particular (Boussicault \cite[Prop. 4.7.2]{TheseBoussicault}), when $\mu=\emptyset$,
so that $D$ is the Ferrers diagram for a partition\footnote{Such bipartite graphs were
called {\it $\lambda$-complete} in \cite{TheseBoussicault}, and sometimes appear in the literature under
the name {\it Ferrers graphs}.} $\lambda$,
this can be rewritten
$$
\Psi_{P_D}(\xx) = \frac{ \symm_{\hat{w}}(\xx,\yy)} { \symm_{w}(\xx,\yy)   }
$$
where $\symm_{w}(\xx,\yy), \symm_{\hat{w}}(\xx,\yy)$ are the double Schubert polynomials for
the dominant permutation $w$ having Lehmer code $\lambda=(\lambda_1,\ldots,\lambda_r)$,
and the vexillary permutation $\hat{w}$ having Lehmer code 
$\hat{\lambda}:=(0,\lambda_2-1,\ldots,\lambda_r-1)$.
}

\subsection{Forests}

In his treatment of the character table for the symmetric group $\symm_n$, 
D.E. Littlewood \cite[p. 85]{Littlewood} 
used the fact that the {\it antichain} poset $P=\varnothing$, having no
order relations on $\{1,2,\ldots,n\}$ and whose set of linear extensions $\LLL(\varnothing)$ 
is equal to all of $\symm_n$, satisfies
\begin{equation}
\label{Littlewoods-identity}
\Phi_{\varnothing}(\xx) = \frac{1}{x_1 x_2 \cdots x_n}.
\end{equation}
The following generalization appeared more recently in \cite{CHNT}.
Say that a poset $P$ is a {\it forest} if every element is covered by
at most one other element.  
\vskip .1in
\noindent
{\bf Theorem C.} 
(Chapoton, Hivert, Novelli, and Thibon \cite[Lemma 5.3]{CHNT})
{\it
For any forest poset $P$,
$$
\Phi_P(\xx) = \frac{1}{\prod_{i=1}^n \left( \sum_{j \leq_P i} x_j\right)}.
$$
}

\subsection{The geometric perspective of cones}

Our first new perspective on these results views $\Psi_P(\xx) , \Phi_P(\xx)$
as instances of a well-known valuation on convex polyhedral cones $K$ in a
Euclidean space $V$ with inner product $\langle \cdot , \cdot \rangle$
$$
s(K;\xx):=\int_K e^{-\langle \xx, v \rangle} dv.
$$
One can think of $s(K;\xx)$ as the multivariable {\it Laplace transform}
applied to the $\{0,1\}$-valued characteristic
function of the cone $K$.  After reviewing the properties of this valuation in 
Section~\ref{Cone-section}, we use these to establish that 
$$
\begin{aligned}
\Psi_P(\xx) &= s(\Kroot_P;\xx) \\
\Phi_P(\xx) &= s(\Kweight_P;\xx) 
\end{aligned}
$$
where $\Kroot_P, \Kweight_P$ are two cones naturally
associated to the poset $P$ as follows:
$$
\begin{aligned}
\Kroot_P  &=\RR_+\{ e_i - e_j : i <_P j\}\\
\Kweight_P&=\{x \in \RR_+^n: x_i \geq x_j \text{ for }i <_P j \},
\end{aligned}
$$
$\RR_+$ denotes the nonnegative real numbers.  
In  Sections~\ref{skew-diagram-section} and
\ref{extreme-rays-section}, this identification is used, together with the properties
of $s(K;\xx)$ from Section~\ref{Cone-section},
to give simple geometric proofs underlying Theorems B and C above.

\subsection{The algebraic perspective of Hilbert series}

One gains another useful perspective when the cone
$K$ is {\it rational} with respect to some lattice $L$ inside $V$, which holds for both
$\Kroot_P, \Kweight_P$.  This
allows one to compute a more refined valuation, the multigraded {\it Hilbert series}
$$
\Hilb(K \cap L;\xx) := \sum_{ v \in K \cap L\ } e^{\langle \xx, v \rangle }
$$
for the affine semigroup ring $k[K \cap L]$ with coefficients in any field $k$.  
As discussed in Section~\ref{total-residue-section} below, it
turns out that $\Hilb(K \cap L;\xx)$ is a meromorphic function of $x_1,\ldots,x_n$,
whose Laurent expansion begins in total degree $-d$, where $d$ is the dimension
of the cone $K$, with this lowest term of total degree $-d$ equal to $s(K;\xx)$,
up to a predictable sign. This allows one to algebraically analyze the 
ring $k[K \cap L]$, compute its Hilbert series, and thereby recover $s(K;\xx)$.  

For example, in Section~\ref{notch-section}, 
it will be shown that Theorem A by Greene is the
reflection of a {\it complete intersection} 
presentation for the affine semigroup ring of $\Kroot_P$
when $P$ is a strongly planar poset, having generators indexed by
the edges in the Hasse diagram of $P$, and relations among the generators 
indexed by the bounded regions $\rho$.

As another example, in Section~\ref{P-partition-section}, 
it will be shown that Theorem C, along
with the ``maj'' hook formula for forests due to 
Bj\"orner and Wachs \cite[Theorem 1.2]{BjornerWachs} 
are both consequences of an easy Hilbert series 
formula (Proposition~\ref{principal-ideal-equation} below) related to $\Kweight_P$ 
when $P$ is a forest.

\section{Cones and valuations}
\label{Cone-section}

\subsection{A review of cones}

We review some facts and terminology about polyhedral cones;
see, e.g., \cite[Chapter 7]{MillerSturmfels}, \cite[\S4.6]{Stanley-EC1} for background.

Let $V$ be an $n$-dimensional vector space over $\RR$.
A linear function $\ell$ in $V^*$ has as zero set a hyperplane $H$ 
containing the origin, and defines a closed halfspace
$H^+$ consisting of the points $v$ in $V$ with $\ell(v) \geq 0$.
A {\it polyhedral cone} $K$ (containing the origin $0$)
in $V$ is the intersection 
$
K=\bigcap_i H_i^+
$
of finitely many linear halfspaces $H_i^+$,
or alternatively the {\it nonnegative span}
$
K=\RR_+\{u_1,\ldots,u_N\}
$
of finitely many generating vectors $u_i$ in $V$.  
Its {\it dimension}, denoted $\dim_\RR K$,
is the dimension of the smallest linear subspace that contains it.
One says $K$ is {\it full-dimensional} if $\dim_\RR K =n=\dim_\RR V$.

Say that $K$ is {\it pointed} if it contains no lines.  In this case,
if $\{u_1,\ldots,u_N\}$ are a minimal set of vectors for which $K=\RR_+\{u_1,\ldots,u_N\}$,
then the $u_i$ are said to span the {\it extreme rays} $\RR_+u_i$ of $K$; these
rays are unique, although the choice of vectors $u_i$ are unique only up to positive scalings.

Say that $K$ is {\it simplicial} if its extreme rays are spanned by a
linearly independent set of vectors $\{u_1,\ldots,u_N\}$, so that
$N=\dim_\RR K \leq n$.

In the dual space $V^*$ one has the {\it dual} or {\it polar cone}
$$
K^*:=\{ x \in V^*: \langle x,v \rangle \geq 0 \text{ for all }v \in K\}.
$$
The following facts about duality of cones are well-known:
\begin{enumerate}
\item[$\bullet$]
Under the identification $(V^*)^*=V$, one has $(K^*)^*=K$.  
\item[$\bullet$]
A cone $K$ is pointed (resp. full-dimensional) 
if and only if its dual cone $K^*$ is full-dimensional (resp. pointed).  
\item[$\bullet$]
A cone $K$ is simplicial
if and only if its dual cone $K^*$ is simplicial.
\end{enumerate}

\subsection{The Laplace transform valuation}

Choose a basis $v_1,\ldots,v_n$ for $V$ and
dual basis $x_1,\ldots,x_n$ for $V^*$.  
Then the polynomial functions $\QQ[V]$ on $V$ are 
identified with the symmetric/polynomial algebras
$
\Sym(V^*) \cong \RR[x_1,\ldots,x_n]
$
and the rational functions $\QQ(V)$ on $V$
with the field of fractions
$\QQ(x_1,\ldots,x_n)$.

In order to consider integrals on $V$, let $dv=dv_1 \cdots dv_n$ 
denote Lebesgue measure on $\RR^n \cong V$ using the basis $v_1,\ldots,v_n$ for
this identification.

The following proposition defining our first valuation
is well-known;  see ,e.g., \cite[Proposition 2.4]{Barvinok},
\cite[Proposition 5]{BerlineVergne}.

\begin{proposition}
\label{integral-valuation-prop}
There exists a unique assignment of a rational function
$\sss(K;\xx)$ lying in $\QQ(V)=\QQ(x_1,\ldots,x_n)$ to each 
polyhedral cone $K$, having the following properties:
\begin{enumerate}
\item[(i)] $\sss(K;\xx)=0$ when $K$ is not pointed.
\item[(ii)] $\sss(K;\xx)=0$ when $K$ is not full-dimensional.
\item[(iii)] When $K$ is pointed and full-dimensional, 
for each $\xx$ in the dual cone $K^*$ the improper integral 
$\int_K e^{-\langle \xx, v \rangle} dv$ converges, to the value 
given by the rational function $\sss(K;\xx)$.
\item[(iv)] When $K$ is pointed and full-dimensional,
with extreme rays spanned by $\{u_1,\ldots,u_N\}$,
the rational function $\sss(K;\xx)$ can be written with smallest
denominator
$
\prod_{i=1}^N \langle \xx, u_i \rangle.
$
\item[(v)] In particular, when $K$ is full-dimensional 
and simplicial, with extreme rays spanned by $\{u_1,\ldots,u_n\}$, then
$$
\sss(K;\xx) = 
  \frac{ | \det [ u_1,\ldots,u_n ] | }
       { \prod_{i=1}^n \langle \xx, u_i \rangle }.
$$
\item[(vi)] The map $\sss(-;\xx)$ is a {\it solid valuation}, that is,
if there is a linear relation $\sum_{i=1}^t c_i \chi_{K_i} = 0$
among the characteristic functions $\chi_{K_i}$ of the cones $K_i$,
there will be a linear relation 
$$
\sum_{i: \dim_\RR K_i = n } c_i \sss(K_i;\xx) = 0.
$$
\end{enumerate}
\end{proposition}

\subsection{The semigroup ring and its Hilbert series}
\label{semigroup-ring-section}
Now endow the $n$-dimensional real vector space $V$ with a distinguished
lattice $L$ of rank $n$, and assume that the chosen basis $v_1,\ldots,v_n$
for $V$ is also a $\ZZ$-basis for $L$.  

Say that the polyhedral cone $K$ is {\it rational}
with respect to $L$ if one can express 
$K=\RR_+\{u_1,\ldots,u_N\}$ for some elements
$u_i$ in $L$.  The subset $K \cap L$ together with its additive
structure inherited from addition of vectors in $V$ is then called an {\it affine semigroup}.
Our goal here is to describe how one can approach the computation
of the previous valuation $\sss(K;\xx)$ for pointed cones $K$ through the calculation
of the finely graded {\it Hilbert series} for this affine semigroup:
$$
\Hilb(K \cap L;\xx) := \sum_{ v \in K \cap L\ } e^{\langle \xx, v \rangle }.
$$
One should clarify how to interpret this infinite series, as it lives in several 
ambient algebraic objects.  Firstly, it lies in the abelian group $\ZZ\{\{L\}\}$
of all formal combinations 
$$
\sum_{v \in L} c_v \,\, e^{\langle \xx, v \rangle }
$$
with $c_v$ in $\ZZ$, in which there are 
no restrictions on vanishing of the coefficients $c_v$.  This set $\ZZ\{\{L\}\}$
forms an abelian group under addition, but is not a ring.  However it contains the
{\it Laurent polynomial ring} 
$$
\ZZ[L] \cong  \ZZ[X_1^{\pm 1},\ldots,X_n^{\pm n}]
$$
as the subgroup where only finitely many of the $c_v$ are allowed to be
nonzero, 
using the identification via the exponential change of variables
\begin{equation}
\label{exponential-change-of-variables}
X_i =e^{\langle \xx,v_i \rangle},\text{ so that }
X_1^{c_1}\cdots X_n^{c_n}= X^v = e^{\langle \xx, v \rangle} \text{ if  }v:=\sum_{i=1}^n c_i v_i.
\end{equation}
Furthermore, $\ZZ\{\{L\}\}$ forms a module over this subring $\ZZ[L]$.
One can also define the $\ZZ[L]$-submodule of
{\it summable} elements (see \cite[Definition 8.3.9]{MillerSturmfels}), namely those $f$ in $\ZZ\{\{L\}\}$ for which there exists
$p,q$ in $\ZZ[L]$ with $q \neq 0$ and $q\cdot f = p$.
In this situation, say that $f$ {\it sums to} $\frac{p}{q}$ as an element of the fraction
field 
$$
\QQ(L) \cong \QQ(X_1,\ldots,X_n).
$$
General theory of affine semigroups (see, e.g., \cite[Chapter 8]{MillerSturmfels}) says that
for a rational polyhedral cone $K$ and the semigroup $K \cap L$,
the Hilbert series $\Hilb(K \cap L;\xx)$ is always summable.  More precisely,
\begin{enumerate}
\item[$\bullet$]
when $K$ is not pointed, 
$\Hilb(K \cap L;\xx)$ sums to zero.  This is because $K$ will not only
contain a line, but also an $L$-rational line, and then
any nonzero vector $v$ of $L$ lying on this line will 
have $(1-e^{\langle \xx,v \rangle}) \cdot \Hilb(K \cap L;\xx)=0$.
\item[$\bullet$]
when $K$ is pointed and $\{u_1,\dots,u_N\}$ are vectors in $L$ that
span its extreme rays, then one can show that 
$$
\left( \prod_{i=1}^N (1-e^{\langle \xx, u_i \rangle}) \right)
\cdot \Hilb(K \cap L;\xx)
$$
always lies in $\ZZ[L]$.
\end{enumerate}

In fact, one has the following analogue of Proposition~\ref{integral-valuation-prop};
see, e.g., \cite[Proposition 4.4]{Barvinok}, \cite[Theorem 3.1]{BarvinokPommersheim},
\cite[Proposition 7]{BerlineVergne}.

\begin{proposition}
\label{integer-points-valuation-prop}
Let $V$ be an $n$-dimensional vector space $V$.
Let $L$ be the sublattice in $V$ with $\ZZ$-basis $v_1,\ldots,v_n$, and
$V^*$ the dual space, with dual basis $x_1,\ldots,x_n$.

Then there exists a well-defined and 
unique assignment of a rational
function $\HHH(K;\XX)$ lying in $\QQ(X_1,\ldots,X_n)$
to each $L$-rational polyhedral cone $K$,
having the following properties:

\begin{enumerate}

\item[(i)] $\HHH(K;\XX)=0$ when $K$ is not pointed.

\item[(ii)] When $K$ is pointed, 
the Hilbert series $\Hilb(K \cap L;\xx)$ sums to the element
$\frac{p}{q}=\HHH(K;\XX)$, considered as a rational function lying in
$\QQ(L)$.

\item[(iii)] When $K$ is pointed and full-dimensional, 
for each $\xx$ in the dual cone $K^*$ the infinite sum
$\sum_{ v \in K \cap L\ } e^{\langle \xx, v \rangle }$
converges, to the value given by the exponential substitution 
\eqref{exponential-change-of-variables}
into the rational function $\HHH(K;\XX)$

\item[(iv)] When $K$ is pointed and full-dimensional,
with $\uu=\{u_1,\ldots,u_N\}$
the unique primitive vectors (that is, those lying in $L$ nearest the origin) that span its extreme rays,
the rational function $\HHH(K;\XX)$ can be written with smallest 
denominator
$
\prod_{i=1}^N (1-X^{u_i}).
$

\item[(v)] In particular, if $K$ is simplicial and $\uu:=\{u_1,\ldots,u_d\}$ its set of primitive vectors that span its extreme rays,
define the semi-open parallelepiped
$$
\Pi_{\uu}:=\left\{ \sum_{i=1}^n c_i u_i: 0 \leq c_i < 1 \right\} \subset V.
$$
Then one has
\begin{equation}
\label{simplicial-H-formula}
\HHH(K;\XX) = 
  \frac{ \sum_{u \in \Pi_{\uu} \cap L } X^u }
       { \prod_{i=1}^d (1-X^{u_i}) }.
\end{equation}
\item[(vi)] The map $\HHH(- ;\XX)$ is a valuation:
if there is a linear relation $\sum_{i=1}^t c_i \chi_{K_i} = 0$
among the characteristic functions $\chi_{K_i}$ of a collection of
($L$-rational) cones $K_i$, there will be a linear relation 
$$
\sum_{i=1}^t c_i \HHH(K_i;\XX) = 0.
$$
\end{enumerate}
\end{proposition}

\subsection{Why $\HHH(K;\XX)$ is finer than $\sss(K;\xx)$}
\label{total-residue-section}

When $K$ is an $L$-rational cone, there is a well-known way (see, e.g., \cite{BrionVergneI})
to compute the Laplace transform valuation $\sss(K;\xx)$ from the Hilbert series valuation
$\HHH(K;\XX)$ by a certain linear {\it residue} operation,
which we now explain.

\begin{proposition}
Let $K$ be an $L$-rational pointed cone, with $\{u_1,\ldots,u_N\}$ vectors
in $L$ that span its extreme rays.
Regard $\HHH(K;\XX)$ as a function
of the variables $\xx=(x_1,\ldots,x_n)$ via the exponential substitution
\eqref{exponential-change-of-variables}.

Then $\HHH(K;\XX)$ is meromorphic in $\xx$, of the form
$$
\HHH(K;\XX)=\frac{h(K;\xx)}{\prod_{i=1}^N \langle \xx, u_i \rangle}
$$
where $h(K;\xx)$ is analytic in $\xx$.  

Furthermore, if $d:=\dim_\RR K$, then the multivariate Taylor expansion 
for $h(K;\xx)$ starts in degree $N-d$, that is,
$$
h(K;\xx)= h_{N-d}(K;\xx) + h_{N-d+1}(K;\xx) + \cdots.
$$
where $h_i(K;\xx)$ are homogeneous polynomials of degree $i$,
and the multivariate Laurent expansion for $\HHH(K;\XX)$ starts 
in degree $-d$, that is,
$$
\HHH(K;\XX)= \HHH_{-d}(\xx) + \HHH_{-d+1}(\xx) + \HHH_{-d+2}(\xx) + \cdots.
$$

Lastly, when $K$ is full-dimensional (so $d=n$), then
$$
\sss(K;\xx)=(-1)^n \frac{h_{N-n}(K;\xx)}{\prod_{i=1}^N \langle \xx, u_i \rangle}
)=(-1)^n \HHH_{-n}(\xx)
$$
so that $h_{N-n}(K;\xx)$ is $(-1)^n$ times the numerator for $\sss(K;\xx)$ accompanying
the smallest denominator described in Proposition~\ref{integral-valuation-prop}(iv).
\end{proposition}

\begin{proof}
We first check all of the assertions when $K$ is simplicial, 
say with extreme rays spanned by the vectors $u_1,\ldots,u_d$ in $L$.
In this case, $N=d$ and the exponential substitution of variables
\eqref{exponential-change-of-variables} into \eqref{simplicial-H-formula}
gives
\begin{equation}
\label{simplicial-H-expansion}
\HHH(K;\XX) 
  = \frac{ \sum_{u \in \Pi_{\uu}} e^{\langle \xx, u \rangle} }
                         { \prod_{i=1}^d (1-e^{\langle \xx, u_i \rangle}) } 
  = (-1)^d \frac{\sum_{u \in \Pi_{\uu}} e^{\langle \xx, u \rangle}}
        {\prod_{i=1}^d \langle \xx, u_i \rangle }
                       \prod_{i=1}^d \frac{\langle \xx, u_i \rangle}
                                      { e^{\langle \xx, u_i \rangle}-1 }.
\end{equation}
We wish to be somewhat explicit about the Taylor expansion of each factor 
in the last product within \eqref{simplicial-H-expansion}.  To this end, recall that the function 
$$
\frac{x}{e^x-1} =  \sum_{n \geq 0} B_n \frac{x^n}{n!} 
= 1-\frac{1}{2} x + \frac{1}{12}x^2 - \frac{1}{720}x^4 + \cdots
$$ 
is analytic in the variable $x$, having power series coefficients described by
the {\it Bernoulli numbers} $B_n$.
Consequently, for each $i=1,2,\ldots,d$ the factor 
$
\frac{\langle \xx, u_i \rangle}{e^{\langle \xx, u_i \rangle}-1}
$
appearing in \eqref{simplicial-H-expansion} is analytic in the variables
$\xx=(x_1,\ldots,x_n)$, and has power series expansion that begins with
constant term $+1$.  Note that the sum 
$$
\sum_{u \in \Pi_{\uu}} e^{\langle \xx, u \rangle}
\sum_{u \in \Pi_{\uu}} \left( 1+{\langle \xx, u \rangle}
                         +\frac{1}{2}{\langle \xx, u \rangle}^2+\cdots \right)
$$
is also analytic in $\xx$, having power series expansion that begins with the
constant term $|\Pi_{\uu}|$.  
Thus the expansion in \eqref{simplicial-H-expansion} begins in degree $-d$ with
$$
 (-1)^d \frac{|\Pi_{\uu}|}{\prod_{i=1}^d \langle \xx, u_i \rangle }.
$$
Whenever $K$ is full-dimensional, so that $d=n$, expressing the $u_i$ in coordinates with
respect to a $\ZZ$-basis $e_1,\ldots,e_n$ for $L$,
one has $|\Pi_{\uu}|=|\det(u_1,\ldots,u_n)|$.
Comparison with Proposition~\ref{integral-valuation-prop}(v) then shows that
the proposition is correct when $K$ is simplicial.

When $K$ is pointed but not simplicial, it is well-known 
(see, e.g., \cite[Lemma 4.6.1]{Stanley-EC1}) 
that one can triangulate $K$ as a complex of simplicial subcones $K_1,\ldots,K_t$
whose extreme rays are all among the extreme
rays $u_1,\ldots,u_N$ for $K$.  This triangulation lets one express the characteristic
function $\chi_K$ in the form (cf. \cite[Lemma 4.6.4]{Stanley-EC1}) 
$
\chi_K = \sum_{j} c_j \chi_{K_j}
$
where the $c_j$ are integers, and $c_j=+1$ whenever the cone $K_j$
has the same dimension as $K$.  
Thus by Proposition\ref{integral-valuation-prop}(vi), one has
$$
\HHH(K;\XX) = \sum_{j} c_i \HHH(K_j;\XX),
$$
which shows that 
$
h(K;\xx):= \left( \prod_{i=1}^N \langle \xx, u_i \rangle \right) \HHH(K;\XX)
$
is analytic in $\xx$.  Furthermore after clearing denominators, it gives the expansion
$$
h(K;\xx)
 = \sum_{j} c_j 
   \left(
     \prod_{ \substack{ i: u_i \text{ a ray of }K,\\ 
             \text{but not of }K_j} }  
             \langle \xx, u_i \rangle 
    \right)
    h(K_j;\xx).
$$
Since the simplicial cones $K_j$ have at most $n$ extreme rays, this shows
$h_i(K;\xx)=0$ for $i < N-n$, and that
$$
h_{N-n}(K;\xx)
 = \sum_{j: \dim_\RR K_j = n}
   \left(
     \prod_{ \substack{ i: u_i \text{ a ray of }K,\\ 
             \text{but not of }K_j} }  
             \langle \xx, u_i \rangle 
    \right)
    h_0(K_j;\xx),
$$
using the fact that $c_j=+1$ whenever $\dim_\RR K_j=\dim_\RR K$.
Dividing through by $\prod_{ i=1 }^N \langle \xx, u_i \rangle $, and multiplying
by $(-1)^n$ gives
$$
(-1)^n \frac{h_{N-n}(K;\xx)}{\prod_{ i=1 }^N \langle \xx, u_i \rangle}
 = \sum_{j: \dim_\RR K_j = n} \sss(K_j;\xx) = \sss(K;\xx)
$$
where the first equality uses the simplicial case already proven,
and the last equality uses Proposition~\ref{integral-valuation-prop}(v).
\end{proof}

The linear operator passing from the meromorphic function
$\HHH(K;\XX)$ of $\xx$ to the rational function $\HHH_{-n}(K;\xx)=(-1)^n \sss(K;\xx)$
has been called taking the {\it total residue} in \cite{BrionVergneI}, where other
methods for computing it are also developed.

\subsection{Complete intersections}
\label{complete-intersection-section}

For a pointed $L$-rational polyhedral cone $K$, one 
approach to computing $\HHH(K;\xx)$ (and hence $\sss(K;\xx)$) is
through an algebraic analysis of the affine semigroup $K \cap L$
and its affine semigroup ring 
$$
R:=k[K \cap L]=k\{e^u\}_{u \in (K \cap L)}
$$ 
over some coefficient field $k$.
We discuss this here, with the case where $R$ is a complete intersection being
particularly simple.

For any semigroup elements $u_1,\ldots,u_m$ in $K \cap L$, one can
introduce a polynomial ring $S:=k[U_1,\ldots,U_m]$, and a ring homomorphism
$S  \longrightarrow  R$ sending $U_i \longmapsto e^{u_i}.$
This map makes $R$ into an $S$-module.  
One also has a fine $L$-multigrading on $R$
and $S$ for which $\deg(U_i)=\deg(e^{u_i})=u_i$.  This makes $R$ an $L$-graded
module over the $L$-graded ring $S$.
It is not hard to see that $R$ is a {\it finitely-generated} $S$-module if and only
if $\{u_1,\ldots,u_m\}$ contain at least one vector 
spanning each extreme ray of $K$.  


When $u_1,\ldots,u_m$ generate (not necessarily minimally) the semigroup $K \cap L$,
the map $S \rightarrow R$ is surjective, and its kernel $I$ is often called
the {\it toric ideal} for $u_1,\ldots,u_m$.  

\begin{proposition}( \cite[Theorem 7.3]{MillerSturmfels}, \cite[Lemma 4.1]{Sturmfels} )
\label{toric-ideal-generators}
One can generate the toric ideal 
$
I=\ker(S \rightarrow R)
$
by finitely many $L$-homogeneous elements chosen among the binomials
$
U^\alpha-U^\beta
$
for which $\alpha, \beta \in \NN^m$ and
$\sum_{i=1}^m \alpha_i u_i = \sum_{j=1}^m \beta_j u_j$. $\qed$
\end{proposition}

As $R=S/I$, and because $S$ has Krull dimension $m$
while $R$ has Krull dimension $d:=\dim_\RR K$, the number of generators for the
ideal $I$ is at least $m-d$.  
The theory of Cohen-Macaulay rings says that,
 since the polynomial algebra $S$ is Cohen-Macaulay, 
whenever the ideal $I$ in $S$ can be generated
by {\it exactly} $m-d$ elements $f_1,\ldots,f_{m-d}$
then these elements must form an {\it $S$-regular sequence}:
for each $i \geq 1$, the image of $f_i$ forms a nonzero divisor in the
quotient $S/(f_1,\ldots,f_{i-1})$.  In this case, the presentation 
$R=S/I=S/(f_1,\ldots,f_{m-d})$
is said to present $R$ as a {\it complete intersection}.
A simple particular case of this
 occurs when the toric ideal $I$ is principal,
as in Example~\ref{cone-theory-example}
 and in~Corollary \ref{unicyclic-corollary}.
By a standard calculation 
using the nonzero divisor condition 
(see, e.g., \cite[\S 13.4, p. 264]{MillerSturmfels}) 
one concludes the following factorization 
for $\HHH(K;\XX)$ and $\sss(K;\xx)$.

\begin{proposition}
\label{prop:complete-intersection-consequences}
Let $K$ be a pointed $L$-rational cone for which the associated
affine semigroup ring $R=k[K \cap L]$ can be presented as a complete
intersection
$$
R=S/I=k[U_1,\ldots,U_m]/(f_1,\ldots,f_{m-d})
$$
where $U_i=e^{u_i}$ for some generators $u_1,\ldots,u_m$ of $K \cap L$,
and where $f_1,\ldots,f_{m-d}$ are $L$-homogeneous elements
of $S$ with degrees $\delta_1,\ldots,\delta_{m-d}$.  Then 
$$
\HHH(K;\XX)=\frac{\prod_{i=1}^{m-d} (1-\XX^{\delta_i})}{\prod_{j=1}^m(1-\XX^{u_j})}
$$
and if $d=n$ then
$$
\sss(K;\xx)=\frac{\prod_{i=1}^{m-n} \langle \xx, \delta_i \rangle }
                {\prod_{j=1}^m \langle \xx, u_j \rangle}. \qed
$$
\end{proposition}

\begin{example}
\label{cone-theory-example}
Let $V=\RR^3$ with standard basis $e_1,e_2,e_3$
and let $K$ be the full-dimensional, pointed cone whose extreme
rays are generated by the four vectors
$$
\begin{array}{rccl}
u_1 &=e_1 &      &\\
u_2 &=e_1 &+e_2&\\
u_3 &=e_1 &      &+e_3\\
u_4 &=e_1 &+e_2 &+e_3.
\end{array}
$$
$$
\includegraphics[width=300pt]{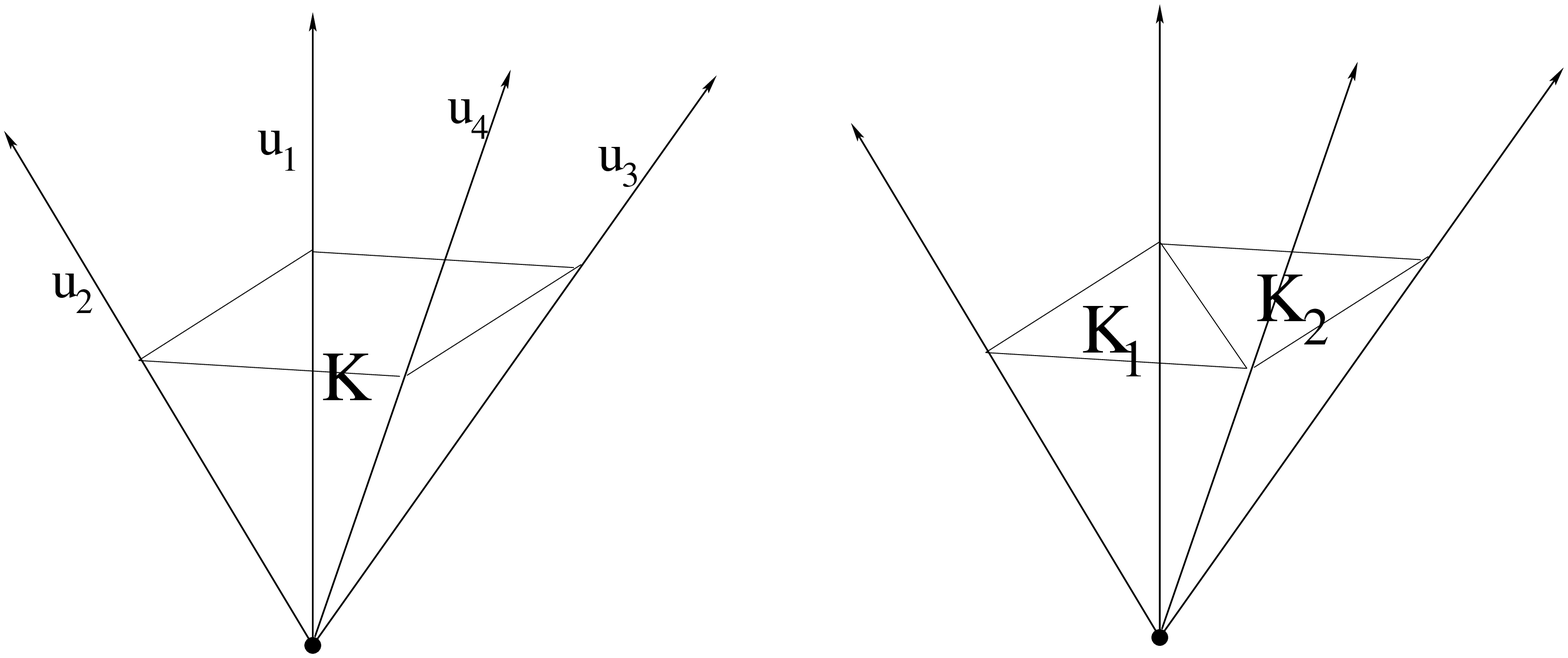}
$$
Note that $K$ is not simplicial, but it can
be expressed as $K=K_1 \cup K_2$ where $K_1, K_2$ are the full-dimensional 
unimodular simplicial cones generated by the two bases for the lattice $L=\ZZ^3$
given by $\{u_1,u_2,u_4\}, \{u_1,u_3,u_4\}$ respectively.
Their intersection $K_1 \cap K_2$ is the $2$-dimensional simplicial 
cone generated by $\{u_1,u_4\}$.

Therefore, applying properties (vi) and then (v)
from Proposition~\ref{integral-valuation-prop}, one can compute
$$
\begin{aligned}
\sss(K;\xx)
   & \overset{(vi)}{=} \sss(K_1;\xx)+\sss(K_2;\xx) \\
   & \overset{(v)}{=} \frac{1}{x_1(x_1+x_2)(x_1+x_2+x_3)} +
         \frac{1}{x_1(x_1+x_3)(x_1+x_2+x_3)} \\
   & = \frac{2x_1+x_2+x_3}{x_1(x_1+x_2)(x_1+x_3)(x_1+x_2+x_3)}.
\end{aligned}
$$

\noindent
Alternatively, one could first compute $\HHH(K,\XX)$ via
Proposition~\ref{integer-points-valuation-prop} (vi) and (v):
\begin{equation}
\label{example-Hilbert-series-computation}
\begin{aligned}
\HHH(K;\XX)
   & \overset{(vi)}{=} \HHH(K_1;\XX)+\HHH(K_2;\XX)-\HHH(K_1 \cap K_2;\XX) \\
   & \overset{(v)}{=} \frac{1}{(1-X_1)(1-X_1 X_2)(1-X_1 X_2 X_3)} \\
   & \qquad       + \frac{1}{(1-X_1)(1-X_1 X_3)(1-X_1 X_2 X_3)} 
          -  \frac{1}{(1-X_1)(1-X_1 X_2 X_3)} \\
   & = \frac{1-X_1^2 X_2 X_3}{(1-X_1)(1-X_1 X_2)(1-X_1 X_3)(1-X_1 X_2 X_3)}.
\end{aligned}
\end{equation}
\noindent
Then one could recover $\sss(K;\xx)$ by first
making the exponential substitution \eqref{exponential-change-of-variables},
then expanding the analytic part $\HHH(K;\XX)$ as a power series in $\xx$, and using this
to extract the homogeneous component $\HHH_{-3}(\xx)$ of degree $-3=-n$:
$$
\begin{aligned}
&\HHH(K;\XX)\\
 &= \frac{1-e^{2x_1+x_2+x_3}}{(1-e^{x_1})(1-e^{x_1+x_2})(1-e^{x_1+x_3})(1-e^{x_1+x_2+x_3})}\\
 &= \frac{1}{x_1(x_1+x_2)(x_1+x_3)(x_1+x_2+x_3)} \\
&\qquad \cdot \left( 1-e^{2x_1+x_2+x_3} \right) 
   \left(\frac{x_1}{1-e^{x_1}}\right)
   \left(\frac{x_1+x_2}{1-e^{x_1+x_2}}\right)
   \left(\frac{x_1+x_3}{1-e^{x_1+x_3}}\right)
   \left(\frac{x_1+x_2+x_3}{1-e^{x_1+x_2+x_3}}\right)\\
&= \frac{-( 2x_1+x_2+x_3 ) + \left(\text{terms of degree at least } 2\right)}
   {x_1(x_1+x_2)(x_1+x_3)(x_1+x_2+x_3)} \\
&\qquad \cdot
   \left( 1 + o(x_1) \right)
   \left( 1 + o(x_1+x_2) \right)
   \left( 1 + o(x_1+x_3) \right)
   \left( 1 + o(x_1+x_2+x_3) \right) \\
&=(-1)^3 \underbrace{\left( \frac{ 2x_1+x_2+x_3 }{x_1(x_1+x_2)(x_1+x_3)(x_1+x_2+x_3)}\right)}_{\sss(K;\xx)} + 
\left(\text{terms of degree at least }-2\right)
\end{aligned}
$$
in agreement with our previous computation.

Alternatively, one can obtain $\HHH(K;\XX)$ and $\sss(K;\xx)$ from
Proposition~\ref{prop:complete-intersection-consequences},
since we claim that $R=k[K \cap L]$ has this 
complete intersection presentation:
$$
R \cong S/I = k[U_1,U_2,U_3,U_4]/(U_1 U_4 - U_2 U_3).
$$
To see this, start by observing that the map 
$$
\begin{aligned}
S=k[U_1,U_2,U_3,U_4] &\overset{\varphi}{\longrightarrow} & R\\
U_i &\longmapsto & e^{u_i} \\
\end{aligned}
$$
is surjective, since $K$ was covered by the two unimodular cones $K_1$ and $K_2$.
Note that there is a unique (up to scaling) linear dependence
\begin{equation}
\label{unique-dependence}
u_1+u_4=u_2+u_3 \quad (=2e_1+e_2+e_3)
\end{equation}
among $\{u_1,u_2,u_3,u_4\}$.  Hence $I=\ker \varphi$ contains the
principal ideal $(U_1 U_4 - U_2 U_3)$.
Furthermore, Proposition~\ref{toric-ideal-generators}
implies that $I$ is generated by binomials of the form $U^\alpha-U^\beta$ where
$\sum_{i=1}^4 \alpha_i u_i = \sum_{j=1}^4 \beta_j u_j$.  Due to the uniqueness of the dependence
\eqref{unique-dependence}, one must have 
$$
\alpha_1=\alpha_4=\beta_2=\beta_3>0 \text{ and }\alpha_2=\alpha_3=\beta_1=\beta_4=0.
$$
Thus $U^\alpha-U^\beta = (U_1U_4)^{\alpha_1}-(U_2U_3)^{\alpha_1}$, which lies in
the ideal $(U_1 U_4 - U_2 U_3)$.  Thus $I=\ker \varphi = (U_1 U_4 - U_2 U_3)$. 
\end{example}

\section{Identifying $\Psi_P$ and $\Phi_P$}
\label{identification-section}

Recall from the introduction that for a poset $P$ on
$\{1,2,\ldots,n\}$ we wish to associate two polyhedral cones.
The first is 
$$
\Kweight_P:=\{x \in \RR_+^n: x_i \geq x_j \text{ for }i <_P j \}
$$
inside the vector space $\RR^n$
with standard basis $e_1,\ldots,e_n$ spanning the appropriate
lattice $\Lweight=\ZZ^n$.  The second is 
$$
\Kroot_P  =\RR_+\{ e_i - e_j : i <_P j\}\\
$$
inside the codimension one subspace $\Vroot \cong \RR^{n-1}$  
of $\RR^n$ where the sum of coordinates $x_1 + \cdots + x_n=0$.  
We consider this
subspace to have Lebesgue measure normalized to make the basis 
$\{e_1-e_2,e_2-e_3,\ldots,e_{n-1}-e_n\}$ for the appropriate
lattice $\Lroot \cong \ZZ^{n-1}$ span a parallelepiped of volume $1$.

\begin{proposition} 
\label{identification-proposition}
For any poset $P$ on $\{1,2,\ldots,n\}$, one has
$$
\begin{array}{rcl}
\Psi_P(\xx) &:=  \sum_{w \in \LLL(P)} 
  w\left(\frac{1}{(x_1-x_2)(x_2-x_3)\cdots(x_{n-1}-x_n)} \right) 
              &=s(\Kroot_P;\xx)  \\
 & \\
\Phi_P(\xx) &:= \sum_{w \in \LLL(P)} 
  w\left(\frac{1}{x_1(x_1+x_2)(x_1+x_2+x_3)\cdots(x_1+\cdots+x_n)} \right) 
              &=s(\Kweight_P;\xx) 
\end{array}
$$
\end{proposition}
\begin{proof}
(cf. Gessel \cite[Proof of Theorem 1]{Gessel})
Proceed by induction on the number of pairs $\{i,j\}$ in $[n]$ that
are {\it incomparable} in $P$.  In the base case where there are no such
pairs, $P$ is a linear order, of the form $P_w$ for some $w$ in $\symm_n$, with 
$\LLL(P_w)=\{w\}$, and the cones
$\Kweight_{P_w}, \Kroot_{P_w}$ are simplicial and unimodular, 
having extreme rays spanned by, respectively,
$$
\begin{array}{rcccl}
&(e_{w(1)}-e_{w(2)},&e_{w(2)}-e_{w(3)},&\ldots,&e_{w(n-1)}-e_{w(n)}) \\
\text{ and }&(e_{w(1)},&e_{w(1)}+e_{w(2)},&\ldots,&e_{w(1)}+e_{w(2)}+\cdots+e_{w(n)}).
\end{array}
$$
Thus Proposition~\ref{integral-valuation-prop}(v) gives the
desired equalities in this case.

In the inductive step, if $i,j$ are incomparable in $P$ then
either order relation $i < j$ or the reverse $j < i$ may be
added to $P$ (followed by taking the transitive closure), to obtain
two posets $P_{i<j}, P_{j<i}$.  Note that 
$$
\LLL(P) = \LLL(P_{i<j}) \sqcup \LLL(P_{j<i})
$$
and hence 
\begin{equation}
\label{incomparable-pair-recurrence}
\begin{aligned}
\Psi_P(\xx) & = \Psi_{P_{i<j}}(\xx) + \Psi_{P_{j<i}}(\xx), \\
\Phi_P(\xx) & = \Phi_{P_{i<j}}(\xx) + \Phi_{P_{j<i}}(\xx). \\
\end{aligned}
\end{equation}
It only remains to show that 
$\sss(\Kroot_P;\xx)$ and $\sss(\Kweight_P;\xx)$
satisfy this same recurrence.
If one introduces into the binary relation $P$ {\it both}
relations $i\leq j$ and $j\leq i$ before taking the transitive closure, 
then one obtains a {\it quasiorder} or {\it preorder} that we denote $P_{i=j}$.
It is natural to also introduce the (non-full-dimensional) 
cone $\Kweight_{P_{i=j}}$ lying inside the hyperplane
where $x_i=x_j$, and the (non-pointed) cone 
$\Kroot_{P_{i=j}}$ containing the line $\RR(e_i-e_j)$.  
One then has these decompositions 
$$
\begin{array}{rcccl}
\Kweight_P &= \Kweight_{P_{i < j}} \cup \Kweight_{P_{j < i}} 
              & \text{ with }
              &\Kweight_{P_{i < j}} \cap \Kweight_{P_{j < i}} 
                  &= \Kweight_{P_{i = j}},    \\
 & & & & \\
\Kroot_{P_{i=j}} &= \Kroot_{P_{i < j}} \cup \Kroot_{P_{j < i}} 
              &  \text{ with }
               &\Kroot_{P_{i < j}} \cap \Kroot_{P_{j < i}} 
                 &= \Kroot_{P}   
\end{array}
$$
leading to these relations among characteristic functions of cones:
\begin{equation}
\label{characteristic-function-relation}
\begin{array}{rcl}
\chi_{\Kweight_P}+ \chi_{\Kweight_{P_{i = j}}}    
 &= \chi_{\Kweight_{P_{i < j}}} + \chi_{\Kweight_{P_{j < i}}}, \\
\chi_{\Kroot_P}+ \chi_{\Kroot_{P_{i = j}}}    
 &= \chi_{\Kroot_{P_{i < j}}} + \chi_{\Kroot_{P_{j < i}}}.
\end{array}
\end{equation}
From this one concludes using Proposition~\ref{integral-valuation-prop}(vi) that
$$
\begin{aligned}
\sss(\Kweight_P;\xx) 
 &= \sss(\Kweight_{P_{i<j}};\xx) + \sss(\Kweight_{P_{j<i}};\xx), \\
\sss(\Kroot_P;\xx) 
 &= \sss(\Kroot_{P_{i<j}};\xx) + \sss(\Kroot_{P_{j<i}};\xx) 
\end{aligned}
$$
since Proposition~\ref{integral-valuation-prop}(i)
implies 
$
\sss(\Kweight_{P_{i=j}};\xx) = \sss(\Kroot_{P_{i=j}};\xx) = 0.
$ 
Comparing with \eqref{incomparable-pair-recurrence}, the result follows
by induction.
\end{proof}

\begin{remark}
\label{duality-relation-remark}
The parallel between the relations 
in \eqref{characteristic-function-relation} is {\it not} a coincidence.
It reflects a general duality 
\cite[Corollary 2.8]{BarvinokPommersheim} relating identities
among characteristic functions of cones $K_i$ and their polar dual
cones $K^*_i$:
\begin{equation}
\label{polar-duality-relation}
\sum_i c_i \chi_{K_i} = 0 
\text{ if and only if }
\sum_i c_i \chi_{K^*_i} = 0.
\end{equation}
While it is not true that the cones $\Kweight_P$ and $\Kroot_P$
are polar dual to each other, this is {\it almost} true,
as we now explain.  

The dual space to the hyperplane $x_1+\cdots+x_n=0$,
which is the ambient space for $\Kroot_P$ is the quotient
space $\RR^n/\ell$ where $\ell$ is the line $\RR(e_1+\cdots+e_n)$.
Thus identities among characteristic functions of cones
$\Kroot_P$ give rise via \eqref{polar-duality-relation},
to identities among the characteristic functions of
their dual cones $(\Kroot_P)^*$ inside this quotient
space.  The cone $\Kweight_P$
maps via the quotient mapping  $\RR^n \rightarrow \RR^n/\ell$
to the dual cone $(\Kroot_P)^*$. Moreover, one can check that the intersection
$\Kweight_P \cap \ell$ is exactly the half-line/ray 
$$
\ell^+:=\RR_+(e_1+\cdots+e_n).
$$
Therefore, identities among characteristic functions of
the cones $(\Kroot_P)^*$ ``lift'' to the same identity among
characteristic functions of the cones $\Kweight_P$.

We are still lying slightly here, since just as
in \eqref{characteristic-function-relation}, one must
not only consider the cones $\Kweight_P, \Kroot_P$ for
posets on $\{1,2,\ldots,n\}$, but also for preposets.
See \cite[\S 3.3]{PostnikovReinerWilliams} for more on this preposet-cone dictionary
for the cones $\Kweight_P$.

We remark also that this duality is the source of our terminology
$\Kroot, \Kweight$ for these cones, as the hyperplane
$x_1+\cdots+x_n=0$ is the ambient space for the {\it root lattice}
of type $A_{n-1}$, while the dual space $\RR^n/\ell$ is the
ambient space for its dual lattice, the {\it weight lattice} of type
$A_{n-1}$.
\end{remark}

\section{Application: skew diagram posets and Theorem B}
\label{skew-diagram-section}

Recall from the introduction that to a {\it skew (Ferrers) diagrams} $D=\lambda/\mu$,
thought of as a collection of points $(i,j)$ in the plane occupying rows $1,2,\ldots,r$
numbered top to bottom, and columns $1,2,\ldots,c$ numbered right to left,
we associate a bipartite poset $P_D$ 
on the set $\{x_1,\ldots,x_r,y_1,\ldots,y_c\}$
having an order relation $x_i <_{P_D} y_j$ whenever $(i,j)$ is a point of $D$.

We wish to prove Theorem B from the introduction, evaluating
$\Psi_{P_D}(\xx)$ for every skew diagram $D$.  Without loss of generality,
we will assume for the remainder of this section that the skew diagram $D$ 
is connected in the sense that its poset $P_D$ is connected; otherwise 
both sides of Theorem B vanish (for the left side, via Corollary~\ref{identify-denominators}, and for the right side because the sum is empty).

We exhibit a known triangulation for the cone $\Kroot_{P_D}$.
The cone $\Kroot_{P_D}$ lives in the codimension one subspace $\Vroot$ of
the product space $\RR^{r+c}=\RR^r \times \RR^c$ with standard basis vectors 
$e_1,\ldots,e_r$ and $f_1,\ldots,f_c$, and dual coordinates
$x_1,\ldots,x_r$ and $y_1,\ldots,y_c$.
Here $\Kroot_{P_D}$ is the nonnegative span of the vectors
$
\{e_i - f_j: (i,j) \in D\}.
$
Note that each of these vectors lies in the following affine hyperplane $H$
of $\Vroot$:
\begin{equation}
\label{affine-hyperplane}
H:=\{ (\xx,\yy) \in \RR^r \times \RR^c: 
x_1+\cdots+x_r=1 \text{ and } y_1+\cdots+y_c=-1 \}.
\end{equation}
Thus it suffices to triangulate the polytope $\mathcal{P}_{D}$, which is
the convex hull of these vectors inside this affine hyperplane $H$.

Consider the skew diagram $D$ as the componentwise partial order 
on its elements $(i,j)$.  One finds that $D$ is a {\it distributive lattice},
in which the meet $\wedge$ and join $\vee$ of two elements
$(i,j), (i',j')$ are their componentwise minimums and maximums:
$$
\begin{aligned}
(i,j) \wedge (i',j') &= ( \min(i,i'), \min(j,j') ) \\
(i,j) \vee (i',j') &= ( \max(i,i'), \max(j,j') ). \\
\end{aligned}
$$
Consequently, by Birkhoff's Theorem 
on the structure of finite distributive lattices
\cite[Theorem 3.4.1]{Stanley-EC1}, the lattice $D$ is isomorphic to the
lattice of order ideals for the subposet $\Irr(D)$ of {\it join-irreducible} 
elements of $D$. 

For any finite poset $Q$, Stanley \cite{Stanley-orderpolytopes}
considered a convex polytope called the {\it order polytope} of $\OOO(Q)$,
which one can think of as the convex hull 
within $\RR^Q$ of the characteristic vectors of order ideals of
$Q$; see \cite[Corollary 1.3]{Stanley-orderpolytopes}.

\begin{proposition}
The convex hull $\mathcal{P}_{D}$ of the vectors $\{e_i - f_j: (i,j) \in D\}$
is affinely isomorphic to the order polytope $\OOO(\Irr(D))$
for the poset $\Irr(D)$.
\end{proposition}
\begin{proof}
Identify the join-irreducibles $(i,j)$ in $\Irr(D)$ with basis vectors
$$
\epsilon_1,\ldots,\epsilon_{r-1},\phi_1,\ldots,\phi_{c-1}
$$
in $\RR^{r-1} \times \RR^{c-1}$ as follows:
\begin{enumerate}
\item[$\bullet$] 
if $(i,j)$ covers $(i-1,j)$, identify $(i,j)$ with $\epsilon_{i-1}$, 
\item[$\bullet$]
if $(i,j)$ covers $(i,j-1)$, identify  $(i,j)$ with $\phi_{j-1}$.
\end{enumerate}
One can then check that a general element $(i,j)$ of $D$
corresponds to an order ideal in $\Irr(D)$ whose elements
are identified with 
$\{\epsilon_1,\ldots,\epsilon_{i-1},\phi_1,\ldots,\phi_{j-1}\}$.
Thus the order polytope $\OOO(\Irr(D))$ is simply the convex hull of vectors
\[ \{\epsilon_1+ \dots + \epsilon_{i-1} +\phi_1+\cdots+\phi_{j-1}: (i,j) \in D \}. \]

The linear morphism
\[ \psi :
\begin{array}{rcl}
\RR^r \times \RR^c & \longrightarrow & \RR^{r-1} \times \RR^{c-1} \\
e_{i} & \longmapsto & \epsilon_1+ \dots + \epsilon_{i-1} \\
f_{j} &\longmapsto & \phi_1+\cdots+\phi_{j-1}
\end{array}\]
restricts to an affine isomorphism
$H \to \RR^{r-1} \times \RR^{c-1}$
sending $e_i-f_j$ to $$ \epsilon_1+\cdots+\epsilon_{i-1} +\phi_1+\cdots+\phi_{j-1}.$$
Therefore, $\psi$ restricts further to an isomorphism between $\mathcal{P}_{D}$ and $\OOO(\Irr(D))$.
\end{proof}

\begin{corollary}
For any skew diagram $D$, the 
cone $\Kroot_{P_D}$ has a triangulation into unimodular
cones $K_\pi$ indexed by lattice paths $\pi$ from $(1,1)$ to $(r,c)$.
Furthermore, the extreme rays of $K_\pi$ are spanned by the vectors
$\{e_i-f_j\}_{(i,j) \in \pi}$.

Consequently, as asserted in Theorem B, one has
$$
\Psi_{P_D}(\xx) 
 = \sum_{\pi} \frac{1}{\prod_{(i,j) \in \pi} (x_i - y_j)}
 = \frac{ \sum_{\pi} \prod_{(i,j) \in D \setminus \pi} (x_i - y_j) }
                      {  \prod_{(i,j) \in D} (x_i - y_j)  }.
$$
In particular, when $D$ is the Ferrers diagram $D$ of a partition $\lambda$, one has
$$
\Psi_{P_D}(\xx) = \frac{ \symm_{\hat{w}}(\xx,\yy)} { \symm_{w}(\xx,\yy)   }
$$
where $\symm_{w}(\xx,\yy), \symm_{\hat{w}}(\xx,\yy)$ are the double Schubert polynomials for
the dominant permutation $w$ having Lehmer code $\lambda=(\lambda_1,\ldots,\lambda_r)$,
and the vexillary permutation $\hat{w}$ having Lehmer code 
$\hat{\lambda}:=(0,\lambda_2-1,\ldots,\lambda_r-1)$.
\end{corollary}

\begin{proof}
Stanley \cite[\S5]{Stanley-orderpolytopes} describes a
triangulation of the order polytope $\OOO(Q)$ whose maximal simplices correspond
to linear extensions $\pi$ of $Q$, or to maximal chains $\pi$ in the
distributive lattice of order ideals $J(Q)$.  For $Q=\Irr(D)$, so that $J(Q)=D$,
these linear extensions $\pi$ correspond to lattice paths from $(1,1)$ to $(r,c)$
in the diagram $D$.  Here the vertices spanning the maximal simplex 
in the triangulation corresponding to $\pi$ are the characteristic vectors of the
order ideals on the chain $\pi$.  

Thus one obtains a corresponding triangulation for the polytope, which
is the intersection of $\Kroot_{P_D}$ with the affine hyperplane in
\eqref{affine-hyperplane}, in which the vertices of the
maximal simplex corresponding to $\pi$ are $\{e_i - f_j: (i,j) \in \pi\}$.
Looking instead at the positive cone 
$K_\pi:=\{e_i - f_j: (i,j) \in \pi\}$ spanned by
these vectors therefore gives a triangulation of the cone $\Kroot_{P_D}$.

The cones $K_\pi$ are unimodular:  one can
easily check, via induction on $r+c$, that for any
lattice path $\pi$ from $(1,1)$ to $(r,c)$, the $\ZZ$-linear span
of the vectors $\{ e_i - f_j\}_{(i,j) \in \pi}$ contains
{\it all} vectors of the form 
$$
\begin{aligned}
e_{i}-e_{j} & \text{ for }1 \leq i \neq j \leq r, \\
f_{i}-f_{j} & \text{ for }1 \leq i \neq j \leq c, \\
e_{i}-f_{j} & \text{ for }1 \leq i \leq r \text{ and }1 \leq j \leq c.\\
\end{aligned}
$$
Therefore by Proposition~\ref{identification-proposition} 
and Proposition~\ref{integral-valuation-prop}(vi), one has
$$
\begin{aligned}
\Psi_{P_D} 
  &= \sss(\Kroot_{P_D};\xx) 
  = \sum_{\pi} \sss(K_\pi;\xx) \\
  &= \sum_{\pi} \frac{1}{\prod_{(i,j) \in \pi} (x_i - y_j)}
  = \frac{ \sum_{\pi} \prod_{(i,j) \in D \setminus \pi} (x_i - y_j) }
                      {  \prod_{(i,j) \in D} (x_i - y_j)  }.
\end{aligned}
$$

When $D$ is the Ferrers diagram of a partition $\lambda$,
this denominator product $\prod_{(i,j) \in D} (x_i - y_j)$
is the double Schubert polynomial
$\symm_{w}(\xx,\yy)$ for the dominant permutation $w$ that
has Lehmer code $\lambda$; 
see, e.g., \cite[\S 9.4]{Lascoux}, \cite[eqn. (6.14)]{Macdonald},
or one can argue similarly to the argument for the numerator sum
given in the next paragraph.

There are various ways to identify the numerator sum
$
\sum_{\pi} \prod_{(i,j) \in D \setminus \pi} (x_i - y_j)
$
as $\symm_{\hat{w}}(\xx,\yy)$.  One way is to check
that each lattice path $\pi$ in $D$ gives rise as follows
to a {\it reduced pipe dream} for $\hat{w}$ in the terminology of 
Knutson and Miller \cite[\S 16.1]{MillerSturmfels}:
the $+$'s occur with the (row,column) indices $(i,j)$ given by the lattice points 
{\it not} visited by $\pi$.  Thus the numerator sum is the 
expansion of $\symm_{\hat{w}}(\xx,\yy)$ as a sum over 
reduced pipe dreams for $\hat{w}$; see 
Fomin and Kirillov \cite[Proposition 6.2]{FominKirillov}, 
or Miller and Sturmfels \cite[Corollary 16.30]{MillerSturmfels}.
\end{proof}

\begin{example}
\label{skew-example}
Consider the skew diagram 
$$
D = (4,4,2)/(1,1,0) = 
\begin{matrix}
\cdot & \bullet & \bullet & \bullet \\
\cdot &  \bullet & \bullet & \bullet \\
\bullet  & \bullet  &      &     
\end{matrix}
$$
whose rows and columns we index as follows.
$$
\begin{matrix}
     & y_4   & y_3  & y_2 & y_1 \\
x_1 & \cdot & (1,3) & (1,2) & (1,1) \\
x_2 & \cdot & (2,3) & (2,2) & (2,1) \\
x_3 & (3,4)  & (3,3)  &      &     
\end{matrix}
$$
Thinking of $D$ as a distributive lattice via the componentwise
order on the labels $(i,j)$, one can label
its $5$ join-irreducibles $\Irr(D)$ by the basis vectors $\epsilon_1,
\epsilon_2,\epsilon_3,\phi_1,\phi_2$ as in the above proof.
$$
\begin{matrix}
 \cdot       & \epsilon_2 & \epsilon_1 & \bullet \\
 \cdot       & \bullet     & \bullet     & \phi_1 \\
 \epsilon_3  & \phi_2  &      &     \\
\end{matrix}
$$
The poset $Q$ of join-irreducible elements of $D$ has the following Hasse diagram.
$$
\includegraphics{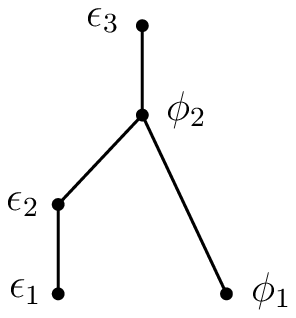}
$$
In this way, the elements of $D$ correspond to the order ideals of $Q$ and to
the vertices of the order polytope $\OOO(\Irr(D))$ as follows.
$$
\begin{matrix}
 \cdot       & \substack{\epsilon_1+\epsilon_2\\} & \substack{\epsilon_1\\} & 0 \\
   & & & \\
 \cdot       & \substack{\epsilon_1+\epsilon_2\\+\phi_1} &  \substack{\epsilon_1\\+\phi_1} & \substack{\phi_1\\} \\
   & & & \\
 \substack{\epsilon_1+\epsilon_2+\epsilon_3\\+\phi_1+\phi_2} &\substack{\epsilon_1+\epsilon_2\\+\phi_1+\phi_2} &      &     \\
\end{matrix}
$$
There are three paths $\pi$ from $(1,1)$ to $(r,c)=(4,3)$, giving rise to the three terms in $\Psi_{P_D}(\xx)$:
$$
\begin{array}{lr}
\begin{matrix}
\cdot & (1,3) & (1,2) & (1,1) \\
\cdot & (2,3) &      &       \\
(3,4)  & (3,3)  &      &     \\
\end{matrix}
& \frac{1}{(x_1-y_1)(x_1-y_2)(x_1-y_3)(x_2-y_3)(x_3-y_3)(x_3-y_4)} \\
 & \\
\begin{matrix}
\cdot & \cdot & (1,2) & (1,1) \\
\cdot & (2,3) &  (2,2)    &       \\
(3,4)  & (3,3)  &      &     \\
\end{matrix}
& +\frac{1}{(x_1-y_1)(x_1-y_2)(x_2-y_2)(x_2-y_3)(x_3-y_3)(x_3-y_4)} \\
 & \\
\begin{matrix}
\cdot & \cdot & \cdot & (1,1) \\
\cdot & (2,3) & (2,2) & (2,1) \\
(3,4)  & (3,3)  &      &     \\
\end{matrix}
& +\frac{1}{(x_1-y_1)(x_2-y_1)(x_2-y_2)(x_2-y_3)(x_3-y_3)(x_3-y_4)}
\end{array}
$$
\end{example}

\section{Extreme rays and Theorem C}
\label{extreme-rays-section}

Our goal here is to identify the extreme rays of the cones $\Kweight_P, \Kroot_P$.
Once achieved, this gives the denominators of $\Psi_P(\xx), \Phi_P(\xx)$,
allows one to decide when the cones are simplicial, leading to Theorem C.

Recall that an {\it order ideal} of a poset $P$ is a subset $J$ of its elements such that, for any
pair $i,j$ of comparable elements ($i \leq_{P} j$), if $j \in J$ then $i \in J$. 

\begin{proposition}
\label{extreme-ray-characterization}
Let $P$ be a poset on $\{1,2,\ldots,n\}$.  
\begin{enumerate}
\item[(i)]
The cone $\Kroot_P$ has extreme rays spanned by $\{ e_i-e_j\}_{i \lessdot_P j}.$ 
\item[(ii)] 
The cone $\Kweight_P$ has extreme rays spanned by the characteristic vectors
$$
e_J:=\chi_J=\sum_{j \in J} e_j
$$
for the connected nonempty order ideals $J$ in $P$.
\end{enumerate}
\end{proposition}

\begin{proof}
For (i), note $\Kroot_P$ is the cone nonnegatively spanned
by $\{e_i-e_j: i <_P j\}$, and since $i <_P j <_P k$
implies 
$$
e_i - e_k = (e_i -e_j) + (e_j - e_k) \in \RR_+\{e_i-e_j, e_j-k\},
$$
its extreme rays must be spanned by some subset of 
$\{e_i - e_j: i \lessdot_P j\}$.  On the other hand, for each covering
relation $i \lessdot_P j$, one can exhibit a linear functional $f$ that
vanishes on $e_i -e_j$ and is strictly negative on the rest of the vectors
spanning $\Kroot_P$ as follows.
Choose a linear extension
$w=(w(1),\ldots,w(n))$ in $\LLL(P)$ such that $i,j$ appear adjacent in the linear order, say
$w(k)=i$ and $w(k+1)=j$
and define the functional $f : \RR^n \to \RR$ by the values
$$
\begin{array}{rcll}
f(e_{w(m)})&=&m&\text{ for }m=1,2,\ldots,k-1 ;\\
f(e_{w(k)})=f(e_i) & =&k=f(e_j) = f(e_{w(k+1)}) ;\\
f(e_{w(m)})&=&m-1&\text{ for }m=k+2,k+3,\ldots,n.
\end{array}
$$

For (ii), note that $\Kweight_P$ is described by the system of inequalities
$$
\begin{cases}
x_i \geq 0 & \text{ for all }i;\\
x_i \geq x_j & \text{ for }i <_P j.
\end{cases}
$$
We first claim that $\Kweight_P$ is the nonnegative
span of characteristic vectors $e_J$ for order ideals $J$ of $P$:
if $x=(x_1,\ldots,x_n)$ lies in $\Kweight_P$, and its coordinates $x_i$
take on the distinct positive values
$
c_1 < c_2 < \cdots < c_t
$
then (setting $c_0:=0$), one has
$$
x  = \sum_{r=1}^t (c_r-c_{r-1}) e_{J_r}
$$
where $J_r$ is the order ideal of $P$ defined by
$$
J_r:=\{ j \in \{1,2,\ldots,n\}: x_j \geq c_r \}.
$$

Furthermore, if an order ideal $J$ of $P$ decomposes into
connected components as $J=\sqcup_i J^{(i)}$,
then each $J^{(i)}$ is itself a (connected) order ideal, and
$e_J = \sum_i e_{J^{(i)}}$.  

Therefore the extreme rays of the cone must be spanned by some subset of the vectors
$e_J$ for connected order ideals $J$.  On the other
hand, for any connected order ideal $J$, one can exhibit the line $\RR e_J$ 
spanned by $e_J$ as the intersection of $n-1$ linearly independent
hyperplanes that come from inequalities valid on $\Kweight_P$ as follows.  
Consider the Hasse diagram for $J$ as a connected graph,
and pick a spanning tree $T$ among its edges.  Then the line $\RR e_J$
is the set of solutions to the system
$$
\begin{cases}
x_i = 0   & \text{ for }i \notin J ;\\
x_i = x_j & \text{ for }i \lessdot_P j \text{ or }i \gtrdot_P j\text{ with }\{i,j\} \in T.
\end{cases}
$$
\end{proof}

Proposition~\ref{integral-valuation-prop} then immediately implies the following.
\begin{corollary}
\label{identify-denominators}
Let $P$ be a poset on $\{1,2,\ldots,n\}$.
\begin{enumerate}
\item[(i)]
If $P$ is disconnected, then the cone $\Kroot_P$ is not full-dimensional,
and $\Psi_P(\xx)=0$.  If $P$ is connected, the cone $\Kroot_P$ is 
full-dimensional, and the smallest denominator for 
$\Psi_P(\xx)$ is $\prod_{i \lessdot_P j}(x_i-x_j)$.
\item[(ii)] 
The cone $\Kweight_P$ is always full-dimensional, and the smallest denominator for 
$\Phi_P(\xx)$ is $\prod_J \left(\sum_{j \in J} x_j \right)$
where the product runs over all connected order ideals $J$ in $P$. $\qed$
\end{enumerate}
\end{corollary}

Theorem C is now simply a consequence of the analysis of the
simplicial cases.
\begin{corollary}
\label{simplicial-characterization}
The cone $\Kroot_P$ is simplicial if and only if the Hasse diagram for $P$
contains no cycles.  In this case it is also unimodular.  Hence 
the Hasse diagram for $P$ is a spanning tree on $\{1,2,\ldots,n\}$, 
if and only if 
$$
\Psi_P(\xx) = \frac{1}{\prod_{i \lessdot_P j}(x_i-x_j)}.
$$

The cone $\Kweight_P$ is simplicial if and only if $P$ is a forest
in the sense that every element is covered by at most one other element.
In this case it is also unimodular.  Hence $P$ is a forest if and only if
$$
\Phi_P(\xx) = \frac{1}{\prod_{i=1}^n \left(\sum_{j \leq_P i} x_j \right)}.
$$
\end{corollary}
\begin{proof}
According to Proposition~\ref{extreme-ray-characterization},
the extreme rays of the cone $\Kroot_P$ are the vectors
$\{e_i - e_j: i \lessdot_P j\}$, which are linearly
independent if and only if there are no cycles in the Hasse diagram
for $P$.  Furthermore, when there are no such cycles, an easy leaf induction
shows that the cone is unimodular.  The rest of the assertions follow.

To analyze $\Kweight_P$, first note that when $P$ is a forest, 
the connected order ideals of $P$ are 
exactly the principal order ideals $P_{\leq i}:=\{j: j \leq_P i\}$ 
for $i=1,2,\ldots,n$.  Not only are their characteristic vectors $e_{P_{\leq i}}$ 
linearly independent, but if one orders the labels $i$ according to any 
linear extension of $P$, one finds that these vectors $e_{P_{\leq i}}$ form
the columns of a unitriangular matrix, which is therefore unimodular.

When $P$ is not a forest, it remains to show that the cone $\Kweight_P$
cannot be simplicial.  There must exist two elements $i,j$ incomparable in $P$ 
whose principal order ideals have nonempty intersection $P_{\leq i} \cap P_{\leq j}$.
Decompose  
$
P_{\leq i} \cap P_{\leq j} = \sqcup_{\ell=1}^t J^{(\ell)}
$
into its connected components $J^{(\ell)}$.  Then each of these
components $J^{(\ell)}$ will be a nonempty connected ideal, as will be $P_{\leq i},
P_{\leq j}$ and their union $P_{\leq i} \cup P_{\leq j}$.  This leads to the following
linear relation:
$$
e_{P_{\leq i}} + e_{P_{\leq j}}
 = e_{P_{\leq i} \cup P_{\leq j}} + \sum_{i=1}^t e_{J^{(\ell)}}.
$$
Since Proposition~\ref{extreme-ray-characterization} 
implies the vectors involved in this relation
all span extreme rays of the cone $\Kweight_P$, the
cone is not simplicial in this case.
\end{proof}

An interesting special case of the preceding result leads
to a special role played by {\it dominant} or {\it $132$-avoiding}
permutations when considering posets of order dimension two,
that is, the subposets of the componentwise order on $\RR^2$.
Bj\"orner and Wachs \cite[Theorems 6.8, 6.9]{BjornerWachs} showed
that $P$ has order dimension two if and only if one can relabel the 
elements $i$ in $[n]$ so that $\LLL(P)$ forms
a principal order ideal $[e,w]$ in the {\it weak Bruhat order} on $\symm_n$.

\begin{corollary}
\label{dominant-permutation-corollary}
When $\LLL(P)=[e,w]$ for some permutation $w$, the cone $\Kweight_P$ is simplicial if and only
if $w$ is $132$-avoiding.
\end{corollary}
\begin{proof}
When $\LLL(P)=[e,w]$, one can check that $P$ has the following order relations: $i <_P j$ 
exactly when $i <_\ZZ j$ and $(i,j)$ are {\it noninversion} values for $w$, that is,
$w^{-1}(i) < w^{-1}(j)$, or $i$ appears earlier than $j$ in the list
notation $(w(1),\ldots,w(n))$.

By Corollary~\ref{simplicial-characterization}, the cone $P$ is not simplicial if and only if $P$ is not a
forest, that is, if and only if there exist $i,j$ which are
incomparable in $P$ and have a common lower bound $h <_P i,j$.  Hence by
the previous paragraph, one must have $h <_\ZZ i$ and $h <_\ZZ j$, with
$h$ appearing earlier than both $i,j$ in the list notation for $w$.
Without loss of generality $i <_\ZZ j$ by reindexing, and then
the incomparability of $i, j$ in $P$ forces $j$ to appear earlier than
$i$ in the list notation.  That is $h <_\ZZ i <_\ZZ j$ occur
in the order $(h,j,i)$ within $w$, forming an occurrence of the pattern $(1,3,2)$.
\end{proof}

\begin{example}
Among the permutations $w$ in $\symm_3$, five out of the six
are dominant or $132$-avoiding; only $w=(1,3,2)$ is not.
It has $[e,w]=\LLL(P)=\{(1,2,3),(1,3,2)\}$, and 
$\Kweight_P$ is the non-simplicial cone 
considered in Example~\ref{cone-theory-example}, having
extreme rays spanned by $\{e_1, \, e_1+e_2, \, e_1+e_3, \, e_1+e_2+e_3\}$,
and
$$
\sss(K;\xx)
    = \frac{2x_1+x_2+x_3}{x_1(x_1+x_2)(x_1+x_3)(x_1+x_2+x_3)}.
$$
\end{example}

\section{$P$-partitions, forests, and the 
Hilbert series for $\Kweight_P$}
\label{P-partition-section}
We digress here to discuss the Hilbert series for the
affine semigroup $K \cap L$ for the cone
$K=\Kweight_P$ inside the lattice $L=\Lweight$.
Analyzing this when $P$ is a forest leads to a common generalization
of both Theorem C and the ``maj'' hook formula for forests of Bj\"orner and
Wachs.

  One can think of as $K \cap L$ as the semigroup of {\it weak $P$-partitions}
in the sense of Stanley \cite[\S 4.5]{Stanley-EC1}, namely functions $f: P \rightarrow \NN$
which are order-reversing: $f(i) \geq f(j)$  for $i <_P j.$
Within this semigroup $K \cap L$, Stanley also considers the semigroup ideal 
$\AAA(P)$ of {\it $P$-partitions} (in the strong sense), 
that is, those order-reversing functions $f:P \rightarrow \NN$ 
which in addition satisfy the strict
inequality $f(i) > f(j)$ whenever $(i,j)$ is in the
{\it descent set} 
$$
\Des(P):=\{(i,j): i \lessdot_P j \text{ and } i>_\ZZ j\}.
$$
The main lemma of $P$-partition theory \cite[Theorem 7.19.4]{Stanley-EC1}
asserts the disjoint decomposition\footnote{This disjoint decomposition
is closely related to the triangulation of $\Kweight_P$ that
appeared implicitly in the proof of Proposition~\ref{identification-proposition},
modelled on Gessel's proof of the main $P$-partition
lemma in \cite[Theorem 1]{Gessel}).}
$$
\AAA(P) = \bigsqcup_{w \in \LLL(P)} \AAA(w).
$$
Equivalently, in terms of the Hilbert series of the semigroup ideal $\AAA(P)$
defined by 
$$
\HHH(\AAA(P);\XX):=\sum_{f \in \AAA(P)} \XX^f
$$
where $\XX^f:=\prod_{i=1}^n X_i^{f(i)}$,
this says that 
\begin{equation}
\label{main-P-partition-equation}
\HHH(\AAA(P);\XX) = \sum_{w \in \LLL(P)} \HHH(\AAA(P_w),\XX).
\end{equation}
This simple equation is more powerful than it looks at first glance.
Define the notation
$\XX^A:=\prod_{j \in A} X_j$ for subsets $A \subset \{1,2,\ldots,n\}$.

\begin{proposition}
For any forest poset $P$ on $\{1,2,\ldots,n\}$, one has
\begin{equation}
\label{principal-ideal-equation}
\HHH(\AAA(P);\XX) = \frac{ \prod_{(i,j) \in \Des(P)} \XX^{P_{\leq i}} } { \prod_{i=1}^n (1-\XX^{P_{\leq i}}) }.
\end{equation}
In particular, \eqref{main-P-partition-equation} becomes
\begin{equation}
\label{fine-hook-formula-for-forests}
\frac{ \prod_{i \in \Des(P)} \XX^{P_{\leq i}} } { \prod_{i=1}^n (1-\XX^{P_{\leq i}}) }
=
\sum_{w \in \LLL(P)} \frac{ \prod_{i: w_i > w_{i+1}} \XX^{\{w_1,w_2,\ldots,w_i\}} } 
                           { \prod_{i=1}^n (1-\XX^{\{w_1,w_2,\ldots,w_i\}}) }.
\end{equation}
\end{proposition}
\begin{proof}
When $P$ is a forest, 
we claim that $\AAA(P)$ is actually a {\it principal ideal} within $K \cap L$,
generated by the $P$-partition $f_0$ for which $f_0(i)$ is the number
of descent edges encountered along the unique path in the Hasse diagram 
from $i$ to a maximal element of $P$.  Alternatively $f_0$ is the sum
of characteristic functions of the subtrees $P_{\leq i}$ for which
one has $(i,j)$ in $\Des(P)$ (here $j$ is the unique element covering $i$ in $P$).
In other words, $\AAA(P)=f_0+K \cap L$, and consequently,
$$
\HHH(\AAA(P);\XX) 
= \XX^{f_0} \cdot \HHH(K;\XX)
= \left( \prod_{(i,j) \in \Des(P)} \XX^{P_{\leq i}} \right) 
    \cdot \HHH(K;\XX) .
$$
But then Corollary~\ref{simplicial-characterization} implies that $K \cap L$
is a unimodular cone having extreme rays spanned by the characteristic vectors of the
subtrees $P_{\leq i}$, and hence
\begin{equation}
\label{unimodular-hilb-consequence}
\HHH(K;\XX) = \prod_{i=1}^n (1-\XX^{P_{\leq i}}).
\end{equation}
The rest follows from the observation that when one considers a permutation $w$
as a linearly ordered poset $P_w$ 
having $w(1) <_{P_w} \cdots <_{P_w} w(n)$, it is an example of a forest, 
in which $P_{\leq i}=\{w(1),w(2),\ldots,w(i)\}$. 
\end{proof}

This has two interesting corollaries.  The first is that by applying the total residue
operator discussed in Section~\ref{total-residue-section} 
to \eqref{unimodular-hilb-consequence}, 
one obtains a second derivation of Theorem C.

The second is that by setting $X_j=q$ for all $j$ in equation 
\eqref{fine-hook-formula-for-forests}, 
one immediately deduces the major index $q$-hook formula for forests of 
Bj\"orner and Wachs \cite[Theorem 1.2]{BjornerWachs}:
\begin{corollary}
When $P$ is a forest,
$$
\sum_{w \in \LLL(P)} q^{\maj(w)} 
= q^{\maj(P)} \frac{[n]!_q}{ \prod_{i=1}^n [h(i)]_q }
$$
where 
$$
\begin{aligned}
\maj(P)&:=\sum_{(i,j) \in \Des(P)} |P_{\leq i}|,\\
h(i)&:=|P_{\leq i}|,\\
[n]_q&:=\frac{1-q^n}{1-q}=1+q+q^2+\cdots+q^{n-1},\\
[n]!_q&:=[n]_q [n-1]_q \cdots [2]_q [1]_q. \qed
\end{aligned}
$$
\end{corollary}

\section{Generators for the affine semigroups}
\label{semigroup-generators-section}

The two families of cones $\Kroot_P, \Kweight_P$ share a
pleasant property: the generating sets for
their affine semigroups are as small as possible. This will be
used in Section \ref{sect:anlalysis_semigroup}.

\begin{proposition}
\label{Hilbert-bases-for-cones}
For $P$ any poset on $\{1,2,\ldots,n\}$,
both cones $K=\Kroot_P, \Kweight_P$, and
the appropriate lattices $L=\Lroot, \Lweight$
have the affine semigroup $K \cap L$ generated
by the primitive lattice vectors (the vectors nearest
the origin) lying on the extreme rays of $K$.
\end{proposition}
\begin{proof}
It suffices to produce a triangulation
of $K$ into unimodular cones, each of whose extreme rays
is a subset of these extreme rays of $K$.

For $\Kroot_P$, this essentially follows from the fact that the root
system of type $A_{n-1}$ is totally unimodular-- every
simplicial cone generated by a subset of roots $e_i - e_j$
is a unimodular cone.  Thus one can pick such a triangulation
of $\Kroot_P$ into simplicial subcones $K$ introducing no new extreme
rays arbitrarily, as in \cite[Lemma 4.6.1]{Stanley-EC1}.  

For $\Kweight_P$, one must be more careful
in producing a triangulation of $\Kweight_P$
into unimodular cones introducing no new extreme rays\footnote{
It is not clear, {\it a priori}, that every simplicial cone spanned by a subset
of the extreme rays of $\Kweight_P$ is unimodular,
e.g. consider the cone spanned by these three rays: 
$e_1+e_2, \,\, e_1+e_3, \,\, e_2+e_3$.}
.
Proceed as in the proof of 
Proposition~\ref{identification-proposition} via induction on the number
$|\LLL(P)|$ of linear extensions, but using as base cases the situation where $P$
is a forest, so that $\Kweight_P$ is
a unimodular cone by Corollary~\ref{simplicial-characterization}.

In this inductive step, assuming $P$ is not a forest,
there exist two elements $i, j$ which are incomparable in $P$
with a common lower bound $h <_P i,j$.
As in the proof of Proposition~\ref{identification-proposition}, one has
$$
\LLL(P) = \LLL(P_{i<j}) \sqcup \LLL(P_{j<i})
$$
and hence a decomposition
\begin{equation}
\label{Gessel-decomposition}
K_{\LLL(P)} = K_{\LLL(P_{i<j})} \cup K_{\LLL(P_{j<i})}.
\end{equation}
Note that induction applies to both $P_{i<j}$ and $P_{j>i}$
since they have fewer linear extensions.  
By the symmetry between $i$ and $j$, it only remains to
show that the extreme rays of $K_{\LLL(P_{i<j})}$ are a subset of
those for $K_{\LLL(P)}$, or equivalently, that any subset $J \subseteq [n]$
which induces a connected order ideal of $P_{i<j}$ will also induce
a connected order ideal of $P$.

First note that $J$ will also be an order ideal in $P$,
since $P$ has fewer order relations than $P_{i<j}$. Given any two elements
$a,b$ in $J$, there will be a path 
\begin{equation}
\label{path-in-ideal}
a=a_0,a_1,\ldots,a_m=b
\end{equation}
in $J$ where each pair $a_\ell,a_{\ell+1}$ are comparable in $P_{i<j}$.  If any pair
$a_\ell,a_{\ell+1}$ are incomparable in $P$, this means either $a_\ell \leq i$ 
and $j \leq a_{\ell+1}$, or the same holds swapping the indices $\ell, \ell+1$.
In either case, $j$ must also lie in the ideal $J$ of $P_{i<j}$, 
and hence $h$ and $i$ lie in $J$ too.  Thus one
can replace the single step $(a_\ell,a_{\ell+1})$ in the path \eqref{path-in-ideal}
with the longer sequence $(a_\ell,i,h,j,a_{\ell+1})$ of steps, or the same swapping the indices
$\ell, \ell+1$.
\end{proof}

\section{Analysis of the semigroup for $\Kroot_P$}
\label{sect:anlalysis_semigroup}

In the following subsections, we focus on the cone
$K=\Kroot_P$ with lattice $L=\Lroot$, and attempt
to analyze the structure of the affine semigroup $K \cap L$,
and its semigroup ring $R=k[K \cap L]$ over a field $k$.
Ultimately this leads to Corollary~\ref{planar-complete-intersection},
giving a complete intersection presentation
for $R$ when the poset $P$ is strongly planar, lifting
Greene's Theorem A from the introduction to a statement about
affine semigroup structure.

\subsection{Generating the toric ideal}

The affine semigroup $R=k[K \cap L]$ is naturally
a subalgebra of a Laurent polynomial algebra
$$
R = k[t_i t_j^{-1}]_{i <_P j} 
\quad \subset \quad 
k[t_1,t_1^{-1},\ldots,t_n,t_n^{-1}].
$$
On the other hand, recall from Proposition~\ref{Hilbert-bases-for-cones} that 
the affine semigroup $K \cap L$ is generated by the primitive
vectors $\{ e_i - e_j: i \lessdot_P j \}$ on its extreme rays.
Therefore one can present $R$ as a quotient via the surjection
$$
\begin{aligned}
S:=k[U_{ij}]_{i \lessdot_P j} & \longrightarrow R \\
U_{ij} &\longmapsto t_i t_j^{-1}.
\end{aligned}
$$
Defining as in Section~\ref{complete-intersection-section}
the toric ideal $I:=\ker(S \rightarrow R)$, one has $R \cong S/I$.

It therefore helps to know generators for $I$ in analyzing $R$,
and trying to compute its Hilbert series.  As in 
Proposition~\ref{toric-ideal-generators},
 $I$ is always generated by certain binomials.
However, there is a smaller generating set of binomials available
in this situation.  

Say that a set of edges $C$ in the (undirected) Hasse diagram for $P$
form a {\it circuit}\footnote{Sometimes these are called {\it simple cycles}.} if they can be directed to form a cycle, and they are
minimal with respect to inclusion having this property.  Having fixed a circuit $C$,
and having fixed one of the two ways to orient $C$ as a directed
cycle, say that an edge $\{i,j\}$ of $C$ having $i \lessdot_P j$ goes
{\it with}  $C$ if $\{i,j\}$ is directed toward $j$ in $C$, and goes {\it against} $C$
otherwise.   Define two monomials 
$$
\begin{aligned}
W(C)&:=\prod_{i \lessdot_P j \text{ with }C} U_{ij}\\
A(C)&:=\prod_{i \lessdot_P j \text{ against }C} U_{ij}
\end{aligned}
$$
and define the {\it circuit binomial}
$$
U(C):=W(C) - A(C).
$$

\begin{proposition}
\label{circuit-generators}
For any poset $P$ on $[n]$, the toric ideal $I=\ker(S \rightarrow R)$
where $S=k[U_{ij}]_{i \lessdot_{P} j}$ is generated by the {\it circuit} binomials
$\{U(C)\}$ as $C$ runs through all circuits of the undirected Hasse diagram of $P$.
\end{proposition}
\begin{proof}
Proposition~\ref{toric-ideal-generators} says $I$ is generated by
binomials of the form
\begin{equation}
\label{generic-binomial}
 \prod_{i \lessdot_P j} U_{ij}^{a_{ij}}
- \prod_{i \lessdot_P j} U_{ij}^{b_{ij}}
\end{equation}
where $a_{ij},b_{ij}$ are nonnegative integers such that
$$
\sum_{i \lessdot_P j} a_{ij} (e_i - e_j) =
\sum_{i \lessdot Pj} b_{ij} (e_j - e_i)
$$
or equivalently
$$
\sum_{i \lessdot_P j} a_{ij} (e_i - e_j) - b_{ij}(e_j-e_i) = 0.
$$
In looking for a smaller set of generators for $I$, note
that one may assume that if $a_{ij} \neq 0$ then $b_{ij} = 0$,
else one could cancel factors of $U_{ij}$ from the binomial
in \eqref{generic-binomial}.  This means that the nonnegative
integers $a_{ij},b_{i,j}$ can be thought of as the multiplicities on a
collection $\CCC$ of directed arcs that either go up or down along edges in $P$, with the
$\CCC$-indegree equalling the $\CCC$-outdegree at every vertex.  Thus $\CCC$ can 
be decomposed into collections supported on various circuits $C_1,\ldots,C_t$ of edges
(allowing multiplicity among the $C_i$).  One then finds that
the binomial \eqref{generic-binomial} lies in the ideal generated by the
circuit binomials $\{U(C_i)\}_{i=1}^t$ using the following calculation
and induction on $t$:
$$
\begin{aligned}
 \prod_{i \lessdot_P j} U_{ij}^{a_{ij}}
- \prod_{i \lessdot_P j} U_{ij}^{b_{ij}} 
& = \prod_{i=1}^t W(C_i) - \prod_{i=1}^t A(C_i) \\
& = \underbrace{\left( W(C_1)-A(C_1) \right)}_{U(C_1)} \prod_{i = 2}^t W(C_i) \\
&\qquad  +A(C_1) \left( \prod_{i=2}^t W(C_i) - \prod_{i=2}^t A(C_i) \right). \qedhere
\end{aligned}
$$
\end{proof}

For example, using this (together with 
Proposition~\ref{prop:complete-intersection-consequences})
allows one to immediately compute 
$\HHH(\Kroot_P;\XX)$ and $\Psi_P(\xx)=\sss(\Kroot_P;\xx)$
in the case where the Hasse diagram of $P$ has only one circuit, as done for 
$\Psi_P$ by other means in \cite{Boussicault} and \cite{BoussicaultFeray}.

\begin{corollary}
\label{unicyclic-corollary}
Let $P$ be a poset whose Hasse diagram has only one circuit $C$. 
Considering the elements on $C$ as a subposet, let $\max(C)$ and $\min(C)$
denote its maximal and minimal elements.

Then the complete intersection presentation 
$
R=k[K \cap L] \cong S/( U(C) )
$ 
implies
$$
\begin{aligned}
\HHH(K_P;\XX) 
&=\left( 1-\prod_{i \in \min(C)}X_i \cdot \prod_{j \in \max(C)}X^{-1}_j \right)
        \frac{1}{\prod_{i \lessdot_P j} (1 - X_i X_j^{-1})} \\
\Psi_P(\xx) &= 
\left( \sum_{i \in \min(C)} x_i - \sum_{j \in \max(C)} x_{j} \right)
\frac{1}{\prod_{i \lessdot_P j} (x_i - x_j)}, 
\end{aligned}
$$
assuming $P$ is connected for the latter formula. $\qed$
\end{corollary}

\subsection{The biconnected component reduction}

Since the ideal $I=\ker(S \rightarrow R)$
is generated by the circuits within the undirected Hasse diagram for $P$,
decomposing the Hasse diagram into
its biconnected components provides a reduction 
in understanding the structure of $R$, which we explain next.

First we recall the notion of biconnected components in
an undirected graph $G=(V,E)$.
Say that two edges are {\it circuit-equivalent} if
there is a circuit $C$ of edges that passes through both. Consider
the equivalence classes $E_i$ of the transitive closure 
of this relation\footnote{Actually, this 
relation is already transitive, although we will not need this here.}.
If $V_i$ is the set of vertices which are at least the 
extremity of one edge in $E_i$ let the {\it biconnected components}
of $G$ be the subgraphs $G_i=(V_i,E_i)$ 

\begin{corollary}
\label{biconnected-component-corollary}
If the Hasse diagram for $P$ has biconnected components $P_1,\ldots,P_t$
(regarding each as the Hasse diagram for a poset $P_\ell$), then
one can express the semigroup ring $R_P$ for $P$ as a tensor 
product of graded $k$-algebras:
$$
R_P  \cong R_{P_1} \otimes_k \cdots \otimes_k R_{P_t}
$$
and therefore
$$
\begin{aligned}
\HHH(\Kroot_P;\XX) &= \prod_{\ell=1}^t \HHH(K_{P_\ell};\XX) ;\\
\Psi_P(\xx) &= \prod_{\ell=1}^t \Psi_{P_\ell}(
\xx).\\
\end{aligned}
$$
\end{corollary}
\begin{proof}
Express $R_P$ as $S/I$. Since every edge of the Hasse diagram lies in
a unique biconnected component $P_i$ ($1 \leq i \leq t$), one
has $S \cong \otimes_{\ell=1}^t S_{P_\ell}$
with $S_{P_\ell}:=k[U_{ij}]_{i \lessdot_{P_\ell} j}$.
Since each circuit $C$ is supported on a set of edges that lies within 
a single biconnected component $P_\ell$, Proposition~\ref{circuit-generators}
implies $I=\bigoplus_{\ell=1}^t I_{P_\ell}$ where $I_{P_\ell}$ is the
toric ideal $\ker(S_{P_{\ell}} \rightarrow R_{P_{\ell}})$.
The first assertion follows, and the remaining assertions follow from the first.
\end{proof}

\begin{remark}
    The argument above works in a more general context.
    Namely, if the ambient vector space $V$, lattice $L$, and cone $K$ have
    compatible direct sum decompositions
    \begin{align*}
        V &= V_1 \oplus \dots \oplus V_\ell,\\
        L &= L_1 \oplus \dots \oplus L_\ell,\\
        K &= K_1 \oplus \dots \oplus K_\ell,
    \end{align*}
    then the semigroup ring $R := k[K \cap L]$ has a tensor product decomposition
    \[R  \cong  R_1 \otimes_k \dots \otimes_k R_\ell, \]
    where $R_i = k[K_i \cap L_i]$ for $i=1,\cdots,\ell$.
\end{remark}

\subsection{Notches and disconnecting chains}
\label{notch-section}

Note that Corollary~\ref{biconnected-component-corollary}
provides a somewhat trivial sufficient condition for $\Psi_P(\xx)$ to factor.
Our goal here is a less trivial such condition on $P$,
including a ring-theoretic explanation of the factorization 
due to disconnecting chains from \cite[Theorem 7.1]{BoussicaultFeray}.
This is provided by the following operation which sometimes applies
to the Hasse diagram for $P$.

\begin{definition}
In a finite poset $P$, say that a triple of elements $(a,b,c)$ forms a
{\it notch of $\vee$-shape} (dually, a {\it notch of $\wedge$-shape}) if 
$a \lessdot_P b,c$ (dually, $a \gtrdot_P b,c$), and 
in addition, $b, c$ lie in different connected components of the 
poset $P \setminus P_{\leq a}$
(dually, $P \setminus P_{\geq a}$).

When $(a,b,c)$ forms a notch of either shape in a poset $P$, say that
the quotient poset $\bar{P}:=P/\{b \equiv c\}$, having one fewer element
and one fewer Hasse diagam edge, 
is obtained from $P$ by {\it closing the notch}, and that $P$ is
obtained from $\bar{P}$ by {\it opening a notch}.

It should be noted that when  $(a,b,c)$ forms a $\vee$-shaped notch,
the two elements $b,c$ have no common upper bounds in $P$.  This eliminates
several pathologies which could occur in the formation of the quotient
poset $\bar{P}=P/\{b \equiv c\}$;  e.g., double edges other than the 
edge $\{a,b\},\{a,c\}$, oriented cycles, creation of a new edge in the quotient
that is the transitive closure of other edges.
\end{definition}

\noindent
For example, in Figure~\ref{notch-figure}, the poset $P_2$ contains
a notch of $\vee$-shape $(3,5,5')$, and the poset $P_1$ is obtained
from $P_2$ by closing this notch.

We state the following result relating $\Kroot_{\bar{P}}, \Kroot_P$
in the case when the notch is $\vee$-shaped; the result for
a $\wedge$-shaped notch is analogous.

\begin{theorem}
\label{notch-theorem}
When $\bar{P}$ is obtained from $P$ closing a $\vee$-shaped notch $(a,b,c)$,
the affine semigroup ring $R_{\bar{P}}$ is obtained from the ring $R_P$ by 
modding out the nonzero divisor $t_a t_b^{-1} - t_a t_c^{-1}$:
\begin{equation}
\label{notch-relation}
R_{\bar{P}} \cong R_P / (t_a t_b^{-1} - t_a t_c^{-1}).
\end{equation}

In particular, 
$$
\begin{aligned}
\HHH(\Kroot_{\bar{P}};\XX) 
&= (1-X_a X_b^{-1}) \left[ \HHH(\Kroot_P;\XX) \right]_{X_b=X_c} \\
\Psi_{\bar{P}}(\xx) 
&= (x_a - x_b) \left[ \Psi_P(\xx) \right]_{x_b=x_c}
\end{aligned}
$$
so that $\Psi_{\bar{P}}(\xx)$ and $\left[ \Psi_P(\xx) \right]_{x_b=x_c}$
have exactly the {\it same} numerator polynomials when written
over the denominator $\prod_{i \lessdot_{\bar{P}} j}(x_i-x_j)$, 
and a complete intersection presentation for $R_P$ leads
to such a presentation for $R_{\bar{P}}$.
\end{theorem}

\begin{example}
\label{disconnecting-chain-example}
Before delving into the proof, we illustrate how 
Theorem~\ref{notch-theorem}, together with some of the 
foregoing results, helps to analyze the ring $R_P,$ as well as the
Hilbert series $\HHH(\Kroot_P;\XX)$, and hence
$\Psi_P(\xx)$.

Consider the posets shown in Figure~\ref{notch-figure}.
As mentioned earlier, $P_1$ is obtained from $P_2$ by closing the
$\vee$-shaped notch $3<5,5'$.  In addition, $P_2$ is obtained
from $P_3$ by closing the $\vee$-shaped notch $1<3,3'$.
Lastly, note that $P_4, P_5$ are the two biconnected components of $P_3$.

\begin{figure}[ht]
$$
\includegraphics[width=250pt]{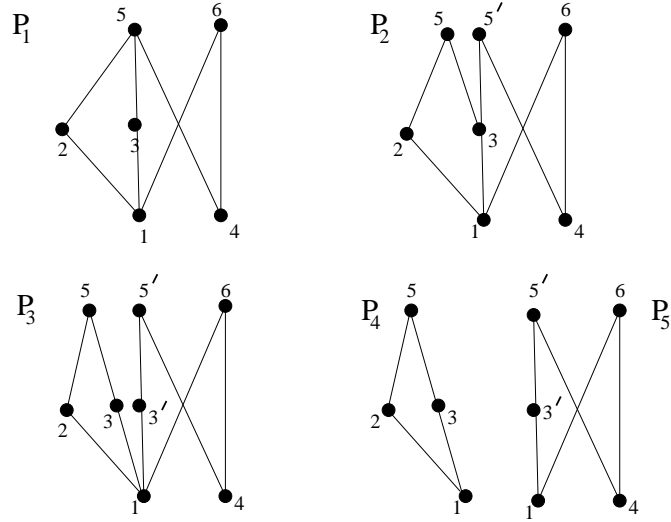}
$$
\caption{Examples of opening and closing notches.}
\label{notch-figure}
\end{figure}

In analyzing $R_{P_1}$, therefore, one can start with $P_4, P_5$,
which each have a unique circuit, and apply 
Corollary~\ref{unicyclic-corollary}
to write down these simple (complete intersection) presentations:
$$
\begin{aligned}
R_{P_4} 
 &\cong k[U_{12},U_{25},U_{13},U_{35}]/(U_{12}U_{25}-U_{13}U_{35}) \\
R_{P_5} 
 &\cong k[U_{13'},U_{16},U_{3'5'},U_{45'},U_{46}]/
              (U_{13'}U_{3'5'}U_{46}-U_{16}U_{45'})
\end{aligned}
$$
Applying Corollary~\ref{biconnected-component-corollary}
yields the following tensor product (complete intersection) presentation
for $R_{P_3}$:
$$
\begin{aligned}
R_{P_3} &\cong R_{P_4} \otimes R_{P_5} \\
         &\cong k[U_{12},U_{25},U_{13},U_{35},U_{13'},U_{16},U_{3'5'},U_{45'},U_{46}] \\
& \qquad \qquad /
           (U_{12}U_{25}-U_{13}U_{35}, \quad 
               U_{13'}U_{3'5'}U_{46}-U_{16}U_{45'}).
\end{aligned}
$$
Applying Theorem~\ref{notch-theorem} to close the notch at
$1<3,3'$ yields the following complete intersection presentation for $R_{P_2}$:
$$
\begin{aligned}
R_{P_2}
   &\cong k[U_{12},U_{25},U_{13},U_{35},U_{16},U_{35'},U_{45'},U_{46}] \\
& \qquad \qquad /
           (U_{12}U_{25}-U_{13}U_{35}, \quad U_{13}U_{35'}U_{46}-U_{16}U_{45'}).
\end{aligned}
$$
Applying Theorem~\ref{notch-theorem} once more to close the notch at
$3<5,5'$ yields the following complete intersection presentation for $R_{P_1}$:
$$
\begin{aligned}
R_{P_1}
   &\cong k[U_{12},U_{25},U_{13},U_{35},U_{16},U_{45},U_{46}] \\
& \qquad \qquad /
           (U_{12}U_{25}-U_{13}U_{35}, \quad U_{13}U_{35}U_{46}-U_{16}U_{45}).
\end{aligned}
$$
Consequently, from Theorem~\ref{prop:complete-intersection-consequences},
one has 
$$
\begin{aligned}
\HHH(\Kroot_{P_1};\XX)
 &=\frac{(1-X_1X_5^{-1})(1-X_1X_4X_5^{-1}X_6^{-1})}
           {\prod_{i \lessdot_{P_1} j}(1-X_iX_j^{-1})} \\
\Psi_{P_1}(\xx)
 &=\frac{(x_1-x_5)(x_1+x_4-x_5-x_6)}
           {\prod_{i \lessdot_{P_1} j}(x_i-x_j^{-1})}.
\end{aligned}
$$

\end{example}

\begin{proof}[Proof of Theorem~\ref{notch-theorem}]
Define $S_P:=k[U_{ij}]_{i \lessdot_P j}$, so that
$$
R_P :=k[\Kroot_P \cap \Lroot] \cong S_P/I_P
$$
where $I_P$ is the kernel of the
map $S_P \rightarrow R_P$ sending $U_{ij}$ to $t_i t_j^{-1}$.

Define a map $S_P \overset{\phi}{\rightarrow} R_{\bar{P}}$ sending
most variables $U_{ij}$ to $t_i t_j^{-1}$, except that
both $U_{ab},U_{ac}$ get sent to $t_a t_b^{-1}$.
We wish to describe the ideal $J:=\ker(S_P \rightarrow R_{\bar{P}})$,
and in particular to show that
\begin{equation}
\label{identify-kernel}
J=I_P+(U_{ab}-U_{ac}).
\end{equation}
This would imply \eqref{notch-relation}:
the map $\phi$ is surjective since it hits
a set of generators for $R_{\bar{P}}$, and hence
$$
\begin{aligned}
R_{\bar{P}} &\cong S_P/J \\
            &= S_P/(I_P + (U_{ab}-U_{ac}))\\
            & \cong \left( S_P/I_P \right) / (\bar{U}_{ab}-\bar{U}_{bc}) \\
            & \cong R_P / (t_a t_b^{-1} - t_a t_c^{-1}).
\end{aligned}
$$

To prove the equality of ideals asserted in \eqref{identify-kernel}, 
one checks that the two ideals are included in each other.  
The inclusion $I_P+(U_{ab}-U_{ac}) \subseteq J$ is not
hard:  both $U_{ab}, U_{ac}$ are sent by $\phi$ to $t_a t_b^{-1}$,
so the binomial $U_{ab}-U_{ac}$ is in the kernel $J$, and since 
circuits $C$ in the directed graph $P$ remain circuits in the quotient
directed graph $\bar{P}$, Proposition~\ref{circuit-generators}
implies the inclusion $I_P \subseteq J$.

  For the reverse inclusion $J \subseteq I_P+(U_{ab}-U_{ac})$,
first note that one can re-interpret the ideal $J$: it is
the toric ideal for the presentation of the semigroup 
$R_{\bar{P}}$ in which the Hasse diagram edge 
$a \lessdot_{\bar{P}} bc$ has been
``doubled'' into two parallel directed edges associated with the
same monomial $t_a t_b^{-1}$, but hit by two variables
$U_{ab}, U_{ac}$ from $S_P$.  Denote by $\bar{P}^+$ this
directed graph obtained from the Hasse diagram for $\bar{P}$
by doubling this edge.  

\begin{figure}[ht]
\includegraphics[height=3cm]{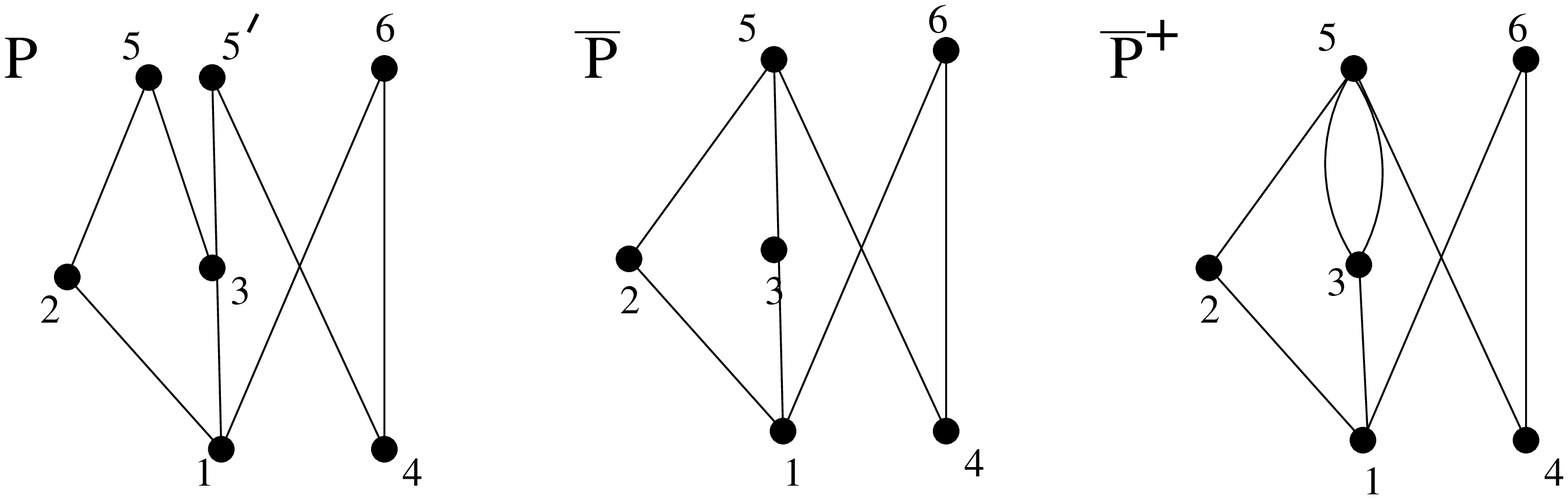}
\caption{An example of $P, \bar{P}, \bar{P}^+$.}
\label{notch-proof-figure}
\end{figure}

\noindent
The analysis from Proposition~\ref{circuit-generators}
then shows that $J$ is generated by the circuit binomials $U(C)$
as $C$ runs through the circuits of $\bar{P}^+$.

It remains to show that for every circuit $C$ in the
directed graph $\bar{P}^+$, the circuit binomial
$U(C)$ lies in $I_P+(U_{ab}-U_{ac})$.  

  If this circuit $C$ in $\bar{P}^+$ does not pass through 
the collapsed vertex $bc$ in $\bar{P}^+$,
then $C$ is also a circuit in $P$, and hence $U(C)$ already lies in $I_P$.

If this circuit $C$ does pass through vertex $bc$, we distinguish two cases.
Consider the partition of the set
$E_{bc}=E_b \sqcup E_c$ of edges
incident to $bc$ in $\bar{P}^+$, where $E_b$ (resp. $E_c$)
is the subset of edges whose preimage in $P$ is incident to $b$ (resp. $c$).
If the two edges of $C$ incident to $bc$ lie in the same set of this partition,
then, as before, $C$ is also a circuit in $P$, and hence $U(C)$ already lies in $I_P$.

Consider now the last case where $C$ does pass through vertex $bc$,
 but the two edges of $C$ incident to $bc$ lie respectively
in $E_b$ and $E_c$. Since $b,c$ lie in different
connected components of $P \setminus P_{\leq a}$,
the circuit $C$ must pass through at least one vertex $d \leq_P a$.  
Use this to create two directed cycles $C_b, C_c$ in $P$:
\begin{enumerate}
\item[$\bullet$] $C_b$ follows $b$ to $d$ along the same path 
$\pi_{bd}$ chosen by $C$, then follows $d$ to $a$ along any saturated chain $\pi_{da}$ in $P$
between them, and finally from $a$ to $b$.
\item[$\bullet$] $C_c$ follows $a$ to $d$ reversing the same saturated
chain $\pi_{da}$, then follows $d$ to $c$ along the same path 
$\pi_{dc}$ chosen by $C$, and finally goes from $c$ to $a$.
\end{enumerate}
One then has the following relation in $S_P$
\begin{equation}
\label{notch-circuit-relation}
\begin{aligned}
U(C) &= U(C_b) \cdot W(\pi_{dc})\\
     &\quad + U(C_c) \cdot A(\pi_{bd}) \\
      &\qquad + (U_{ab}-U_{ac}) 
          \cdot W(\pi_{dc}) 
          \cdot A(\pi_{bd}) 
          \cdot W(\pi_{da})
\end{aligned}
\end{equation}
where for a path $\pi$ of edges in the Hasse diagram one defines monomials
$$
W(\pi)
:=\prod_{\substack{i \lessdot_P j:\\
            i \rightarrow j \text{ appears in } \pi}} U_{i j} 
$$
$$
A(\pi)
:=\prod_{\substack{i \lessdot_P j:\\
            i \leftarrow j \text{ appears in } \pi}} U_{i j}.
$$
The relation \eqref{notch-circuit-relation}
shows that $U(C)$ lies in $I_P+(U_{ab}-U_{ac})$, as desired.

For the remaining assertions, note that since $R_P$ is 
a subalgebra of the Laurent polynomial ring, it is an integral
domain, and therefore $t_a t_b^{-1}-t_a t_c^{-1}$
is a non-zero-divisor of $R_P$. After identifying the grading 
variables $x_b=x_c$, this element $t_a t_b^{-1}-t_a t_c^{-1}$ 
becomes homogeneous of degree $e_a-e_b$.  
\end{proof}

Opening notches in a poset $P$ provides a flexible way
to understand some previously observed factorizations
of the numerator of $\Psi_P(\xx)$, while at the same
time giving information about the semigroup ring
$k[\Kroot_P \cap \Lroot_P]$ and its Hilbert
series.

\begin{example}
One way to explain the factorization of the numerator of
$\Psi_P(\xx)$ for the example from \cite[Figure 2]{BoussicaultFeray}, 
is to successively ``open two notches'', as shown here
$$
\includegraphics[width=200pt]{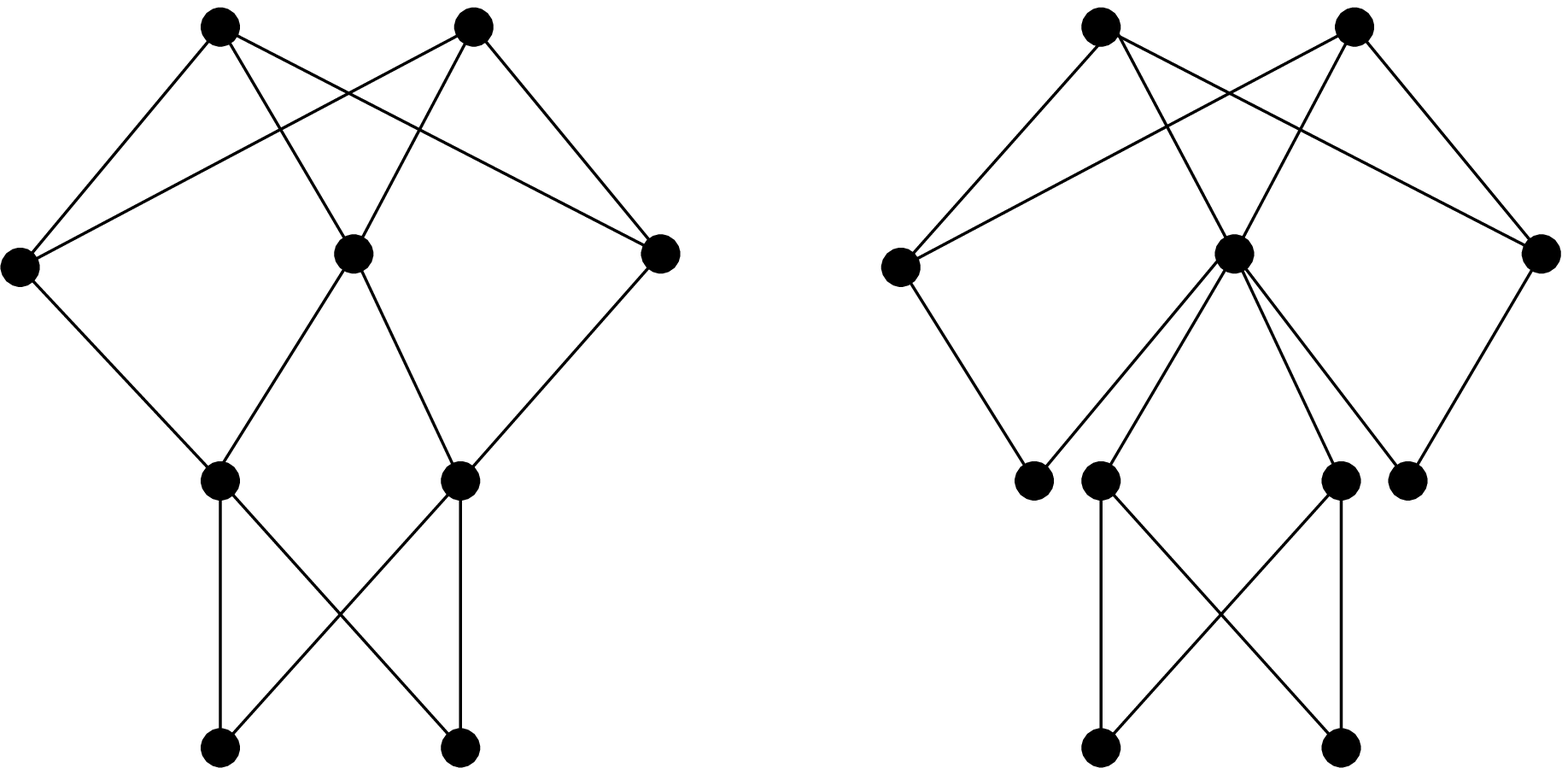}
$$
and then apply Corollary~\ref{biconnected-component-corollary}
to the poset on the right, which has two biconnected components.
\end{example}

\begin{example}
\label{general-disconnecting-chain-example}
In \cite[\S 7]{BoussicaultFeray} it was explained
how a disconnecting chain
$$
\sigma=(p_1 \lessdot_P p_2 \lessdot_P \cdots \lessdot_P p_{t-1} \lessdot p_t)
$$
in $P$, that is, one for which $P \setminus \sigma$ has several
connected components, leads to a factorization of the numerator of
$\Psi_P(\xx)$ into factors indexed by each such
component.  After fixing one of the 
connected components $Q$ of $P \setminus \sigma$,
one can use several operations of opening notches, beginning
with one that creates two elements $p_t,p_t'$ covering $p_{t-1}$,
and continuing down the chain $\sigma$, to ``peel off'' a copy of $Q \sqcup \sigma$ 
until it is attached to $P \setminus Q$ only at the vertex $p_1$.
At this stage use Corollary~\ref{biconnected-component-corollary}, to recover
the factorization of \cite[Theorem 7.1]{BoussicaultFeray}.

We omit a detailed discussion to avoid the use of heavy notation.  
However, Example~\ref{disconnecting-chain-example}
illustrates the principle. 
\end{example}

Lastly, one can use this to deduce a stronger form of Theorem A
from the introduction.  For a strongly planar poset $P$,
and a bounded region $\rho$ of the plane enclosed by
its Hasse diagram, recall that $\min(\rho), \max(\rho)$
denote the $P$-minimum, $P$-maximum elements among
the elements of $P$ lying on $\rho$.  Name the elements
on the unique two
maximal chains from $\min(\rho)$ to $\max(\rho)$
that bound $\rho$ as follows:
\begin{equation}
\label{chains-bounding-region}
\begin{aligned}
&\min(\rho)=:i_0 \lessdot_P 
  i_1 \lessdot_P \cdots \lessdot_P i_{r-1} 
\lessdot_P i_r:=\max(\rho) \\
&\min(\rho)=:j_0 \lessdot_P 
  j_1 \lessdot_P \cdots \lessdot_P j_{s-1} 
\lessdot_P j_s:=\max(\rho) \\
\end{aligned}
\end{equation}
Lastly, let $f_\rho$ be the following binomial
in the polynomial algebra $S:=k[U_{ij}]_{i \lessdot_P j}$:
$$
f_\rho:= \prod_{p=1}^r U_{i_{p-1} i_p} -
         \prod_{q=1}^s U_{j_{q-1} j_q}.
$$
In other words $f_\rho$ is the circuit binomial $U(C)$
for the directed circuit $C$ that goes up and down the
two maximal chains in \eqref{chains-bounding-region} bounding $\rho$.
\begin{corollary}
\label{planar-complete-intersection}
For any strongly planar poset $P$ on $\{1,2,\ldots,n\}$,
one has a complete intersection presentation for its semigroup
ring $k[\Kroot_P \cap \Lroot]$ as the quotient $S/I$
where  $S:=k[U_{ij}]_{i \lessdot_P j}$ and $I$ is the ideal
generated by the $\{f_\rho\}$ as $\rho$ runs through all
bounded regions for the Hasse diagram of $P$.

Consequently,
$$
\begin{aligned}
\HHH(\Kroot_P;\XX) 
&= \frac{\prod_{\rho} (1-X_{\min(\rho)}X^{-1}_{\max(\rho)}) }
                {\prod_{i \lessdot_P j} (1-X_i X_j^{-1}) } \\
\Psi_P(\xx)
&= \frac{\prod_{\rho} (x_{\min(\rho)}-x_{\max(\rho)})}
                {\prod_{i \lessdot_P j} (x_i-x_j)}
\end{aligned}
$$
where the last equality assumes that $P$ is connected.
\end{corollary}

\begin{proof}
Use induction on the number of bounded regions $\rho$.
In the base cases where there are no such regions or
one such region, apply 
Corollary~\ref{simplicial-characterization} or 
\ref{unicyclic-corollary}, respectively.

In the inductive step, find a disconnecting chain for $P$
that separates at least two bounded regions,
as in \cite[Proposition 7.4]{BoussicaultFeray}.
Use Proposition~\ref{notch-theorem} repeatedly to
open notches down this chain, until the resulting poset has two
biconnected components attached at one vertex of
the chain, and apply Corollary~\ref{biconnected-component-corollary},
as in Example~\ref{general-disconnecting-chain-example}.
\end{proof}

\section{Reinterpreting the main transformation}
\label{Poinconnage}

  Our goal in this final section is to reinterpret geometrically
a very flexible identity that was used to deduce
most of the results on $\Psi_P(\xx)$ in \cite{BoussicaultFeray},
and called there the {\it main transformation}:
\vskip .1in
\noindent 
{\bf Theorem.}(\cite[Theorem 4.1]{BoussicaultFeray})
{\it 
Let $C$ be one of the two possible orientations
of a circuit in the Hasse diagram for a poset $P$.
Let $W \subset C$ be the edges of $C$ which are
directed upward in $P$.  Then
\begin{equation}
\label{poinconnage}
\sum_{E \subset W} (-1)^{|E|} \Psi_{P \setminus E}(\xx) = 0
\end{equation}
where $P \setminus E$ is the poset whose Hasse diagram is obtained
from that of $P$ by removing the edges in $E$.
}
\vskip.1in
\noindent

\begin{remark}
In fact, \eqref{poinconnage} was deduced in  
\cite[Theorem 4.1]{BoussicaultFeray} from a geometric identity
equivalent to the following:
\begin{equation}
\label{dual-poinconnage}
\sum_{E \subset W} (-1)^{|E|} \chi_{\Kweight_{P \setminus E}} = 0.
\end{equation}
Using the duality discussed in Remark~\ref{duality-relation-remark}, 
identity \eqref{dual-poinconnage} implies the following geometric identity
underlying \eqref{poinconnage}:
\begin{equation}
\label{poinconnage-geometrically}
\sum_{E \subset W} (-1)^{|E|} \chi_{\Kroot_{P \setminus E}} = 0.
\end{equation}
\end{remark}

\begin{remark}
    In \cite{BoussicaultFeray},
    the identity \eqref{poinconnage} was used to prove some statements
    on $\Psi$ by induction 
on the number of independent cycles (the {\it cyclomatic number})
in the Hasse diagram for $P$:  terms indexed by non-empty
subsets $E$ correspond to posets $P \setminus E$ with
fewer independent cycles.  In the base case for such inductive proofs,
the Hasse diagram is acyclic, and possibly disconnected, so that
either $\Psi_P(\xx)=0$, or 
Corollary~\ref{simplicial-characterization} applies.

  Furthermore, in \cite[section 6]{BoussicaultFeray}, it was shown how
the choice of an embedding of the Hasse diagram of $P$ onto a surface,
together with a rooting at one of its half-edges, leads to a good
a choice of circuits $C$ in the induction.  This expresses
$\Psi_P(\xx)=\sum_i \Psi_{P_i}(\xx)$ for various posets $P_i$ with tree
Hasse diagrams that can be described explicitly in terms of the embedding
and rooting.  Using \eqref{poinconnage-geometrically}, one can show that
this corresponds to an explicit 
triangulation for the cone $\Kroot_P$ into subcones
$\Kroot_{P_i}$, in which each subcone uses no new extreme rays.

Unfortunately, iterating \eqref{dual-poinconnage}
does not in general lead to proofs for results
on $\Phi_P(\xx)$ via induction on cyclomatic number,
as the base cases with no cycles correspond to 
cones $\Kweight_P$ which are not necessarily simplicial; 
see Corollary~\ref{simplicial-characterization}.
\end{remark}

\begin{remark}
Unlike equation \eqref{incomparable-pair-recurrence}, 
this identity \eqref{poinconnage-geometrically}
involves only pointed cones.
\end{remark}

Our goal here is to point out how the geometric statement
\eqref{poinconnage-geometrically} generalizes to other families of 
cones and vectors.  We begin with a geometric generalization of the 
notion of a circuit $C$ in the Hasse diagram for $P$
and its subset of upward edges $W \subset C$.

\begin{definition}
Given two subsets of $W, V$ of vectors in $\RR^d$,
say that $W$  {\it is cyclic\footnote{In the special case
where $V$ is empty, this is the notion of $W$ being a
{\it totally cyclic} collection of vectors from oriented
matroid theory; see \cite[Definition 3.4.7]{BLSWZ}.} 
with respect to $V$}
if there exists a positive linear combination of $W$ lying in $\RR_+ V$, 
that is,  
$
\sum_{w \in W} a_w w = \sum_{v \in V} b_v v
$
for some real numbers $a_w > 0, b_v \geq 0$.
\end{definition}

\begin{example}
Let $C$ be one of the two possible orientations
of a directed circuit in the Hasse diagram for a poset $P$.
Let $W \subset C$ be the edges of $C$ which are
directed upward in $P$.  Then
$\{e_i - e_j: (i,j) \in W\}$ 
is cyclic with respect to the set 
$V:=\{e_i - e_j: i \lessdot_P j, (i,j) \notin W\}$,
due to the relation
$$
\sum_{\substack{i \lessdot_P j:\\ (i,j) \in W}} e_i -e_j 
= \sum_{\substack{i \lessdot_P j:\\ (j,i) \in C \setminus W}} e_i -e_j. 
$$
\end{example}

\noindent
Bearing this example in mind, the following proposition
gives the desired generalization of \eqref{poinconnage} and
\eqref{poinconnage-geometrically}.

\begin{proposition}\label{prop:cyclic_vectors}
For subsets $W,V$ of vectors in $\RR^d$ where
$W$ is cyclic with respect to $V$, 
one has the identity among characteristic vectors of cones
$$
\sum_{B \subset W} (-1)^{|B|} \chi_{\RR_+(V \cup B)} = 0
$$
and therefore 
$$
\sum_{B \subset W} (-1)^{|B|} s(\RR_+(V \cup B);\xx) = 0.
$$
\end{proposition}

\begin{example}
Consider the set of vectors $W=\{w_1,w_2,w_3,w_4\}$ in $\RR^2$
shown below, and let $V$ be the empty set.  The set $W$ is easily seen to be cyclic
with respect to $V$.
$$
\includegraphics[width=100pt]{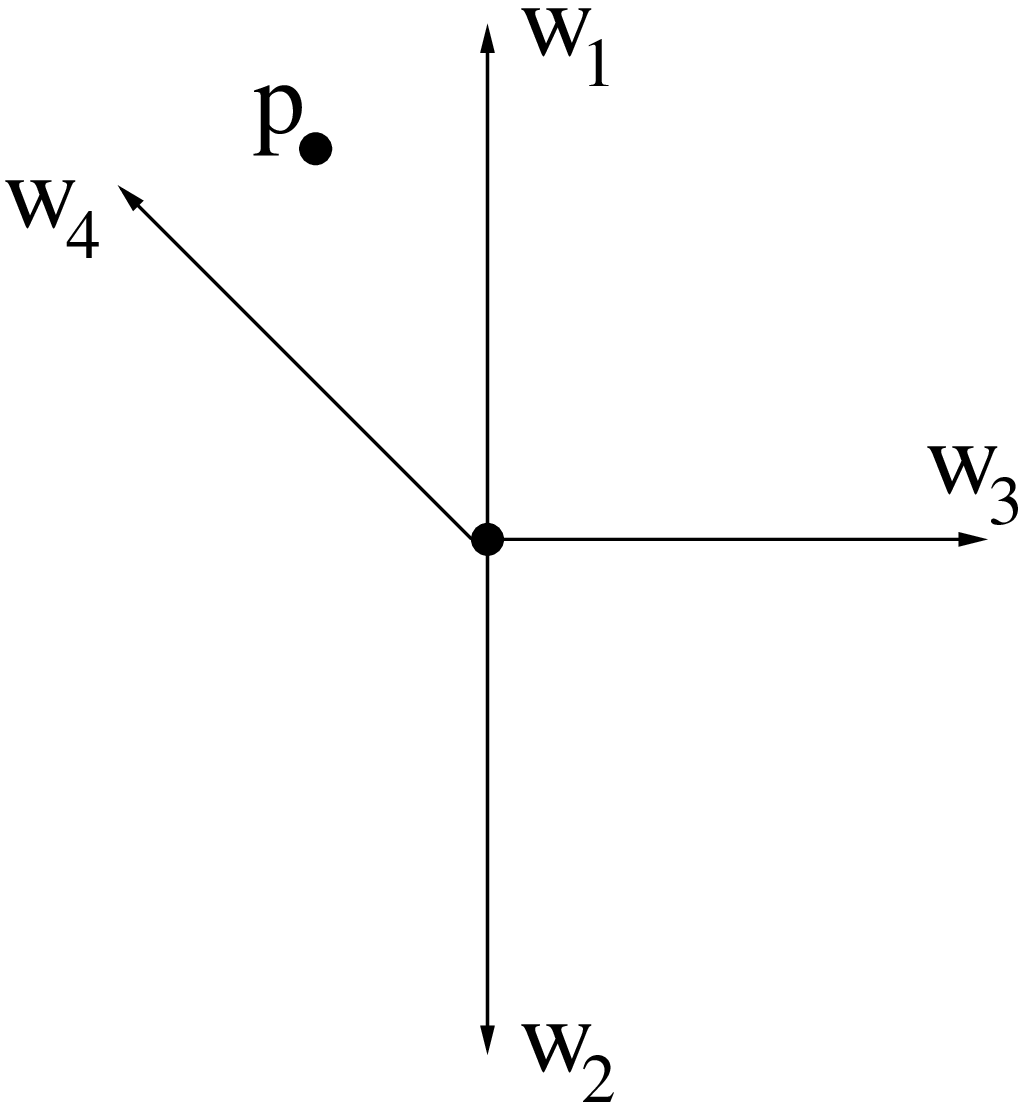}
$$
Consider the point $p$ depicted.  The subsets $B \subset W$ for
which $p$ lies in the cone $\RR_+(V \cup B)$, 
so that $\chi_{\RR_+(V \cup B)}(p)=1$, are
$$
\begin{aligned}
\{w_1,w_4\},
\{w_3,w_4\},
\{w_1,w_2,w_4\},
\{w_1,w_3,w_4\},
\{w_2,w_3,w_4\},
\{w_1,w_2,w_3,w_4\}
\end{aligned}
$$
The sum of $(-1)^{|B|}$ over these sets $B$ vanishes,
as predicted by the proposition.  
However, note that this does not hold for
trivial reasons, e.g., these sets $B$ do not form an interval 
in the boolean lattice.
\end{example}

\begin{proof}[Proof of Proposition \ref{prop:cyclic_vectors}.]
Up to a rescaling of the vectors in $W$, one can assume 
that $u:=\sum_{w \in W} w$ lies in $\RR_+ V$.

One must show that for every point $p \in \RR^d$, one has
\begin{equation}
\label{point-by-point-sum}
\sum_{\substack{B \subset W:\\ p \in \RR_+(V \cup B)}} (-1)^{|B|} = 0.
\end{equation}
If $p$ does not lie in the cone $\RR_+(V \cup W)$, this holds because
the left side is an empty sum.  So without loss of generality
$p$ lies in $\RR_+ (V \cup W)$, meaning that the set
$$
X_p:=\left\{(\aa,\bb) \in \RR_+^W \times \RR_+^V:
p = \sum_{w \in W} a_w w + \sum_{v \in V}b_v v \right\}
$$
is a non-empty convex polyhedral cone inside $\RR^W \times \RR^V$.
Cover $X_p$ by the family of subsets $\{ X_p(w_0) \}_{w_0 \in W}$
defined by
$$
X_p(w_0):=\{(\aa,\bb) \in X_p: a_{w_0}=\min(\aa) \}.
$$
These sets $X_p(w_0)$ are also convex polyhedral subsets, although possibly empty. 
The nerve of this covering of $X_p$ is the abstract simplicial complex 
consisting of all subsets $A \subset W$ for which 
$
\bigcap_{w_0 \in A} X_p(w_0) 
$
is nonempty.  A standard nerve lemma (e.g., \cite[Theorem 10.7]{Bjorner})
implies that the geometric realization of this nerve is homotopy equivalent
to the contractible space $X_p$, and hence its 
{\it (reduced) Euler characteristic} 
$
\sum_{A}
(-1)^{|A|-1}
$
vanishes, where here the sum runs over subsets $A$ with
$\bigcap_{w_0 \in A} X_p(w_0)$ nonempty.
Thus equation \eqref{point-by-point-sum} will follow from this claim:

\begin{quote}
{\bf Claim.}
The set $\bigcap_{w_0 \in A} X_p(w_0)$ is nonempty if and only if
$p$ lies in $\RR_+(V \cup (W \setminus A))$.
\end{quote}

For the ``if'' assertion of the claim, note that if $p$ lies in 
$\RR_+(V \cup (W \setminus A))$, then any expression
$$
p=\sum_{w \in W \setminus A} a_w w + \sum_{v \in V} b_v v
$$
leads to a similar expression
$$
p=\sum_{w \in W} a_w w + \sum_{v \in V} b_v v
$$
by defining $a_{w_0}:=0$ for all $w_0$ in $A$.  
Furthermore, the coefficients in the
latter expression give an element $(\aa,\bb)$ lying
in $\bigcap_{w_0 \in A} X_p(w_0)$.

For the ``only if'' assertion, assuming 
that $\bigcap_{w_0 \in A} X_p(w_0)$ is nonempty,
pick $(\aa,\bb)$ lying in this set.  Thus 
$p=\sum_{w \in W} a_w w + \sum_{v \in V} b_v v$
and one has $\mu:=\min(\aa)=a_{w_0}$ for all $w_0$ in $A$.
Rewriting this as
$$
p=\sum_{w_0 \in A} \mu \cdot w_0 
     + \sum_{w \in W \setminus A} a_w w 
      + \sum_{v \in V} b_v v
$$
and using the fact that $u=\sum_{w \in W} w$ lies in $\RR_+ V$,
one can rewrite
$$
p= \underbrace{\sum_{w \in W \setminus A}  (a_w-\mu) w}_{\in \RR_+ (W\setminus A)}
      \quad + \quad
       \underbrace{\mu \cdot u +\sum_{v \in V} b_v v}_{\in \RR_+ V}.
$$
Therefore $p$ lies in $\RR_+(V \cup (W \setminus A))$.
\end{proof}

\section*{Acknowledgements}
This work began during a sabbatical visit of V.R. to
the Institut Gaspard Monge at the Universit\'e Paris-Est,
and he thanks them for their hospitality.  
He is also grateful to Prof. Michelle Vergne for an 
enlightening explanation of total residues.

This work was finished during a visit of the second author to the University
of Minnesota and he thanks them for the invitation and the welcoming
environment.

The authors also would like to thank an anonymous referee for helpful comments.


\end{document}